\documentclass[10pt]{article}

\usepackage{amssymb, latexsym, amsfonts, pigpen, enumitem, mathrsfs, amsthm, bbm, stackrel, todonotes, mathtools,caption, scalerel}
\usepackage[matrix,arrow,curve]{xy}
\usepackage[letterpaper,top=1.05in, bottom=1.05in, left=1.05in, right=1.05in]{geometry}
\usepackage[colorlinks=true,pagebackref=true]{hyperref}

\newcommand\C{{\mathcal C}}

\newcommand{\D}{\mathcal D}

\newcommand*{\rep}{\mathop{\mathrm {rep}}\nolimits}

\newcommand{\ind}{{\mathrm{ind}}}

\newcommand{\res}{{\mathrm{res}}}

\newcommand{\hull}{{\mathrm{hull}}}
\newcommand{\id}{{\mathrm{id}}}

\newcommand*{\Dd}{\mathop{\mathrm D\kern0pt}\nolimits}

\newcommand*{\Hom}{\mathop{\mathrm {Hom}}\nolimits}
\newcommand*{\End}{\mathop{\mathrm {End}}\nolimits}
\newcommand*{\iEnd}{\mathop{\underline{\mathrm{End}}}\nolimits}
\newcommand*{\iHom}{\mathop{\underline{\mathrm{Hom}}}\nolimits}

\newcommand*{\Coh}{\mathop{\sf {Coh}}\nolimits}
\newcommand*{\CohAlg}{\mathop{\sf {CohAlg}}\nolimits}
\newcommand*{\SAlg}{\mathop{\sf {SAlg}}\nolimits}
\newcommand*{\CSAlg}{\mathop{\sf {CSAlg}}\nolimits}

\newcommand*{\Alg}{\mathop{\sf {Alg}}\nolimits}
\newcommand*{\QCoh}{\mathop{\sf {QCoh}}\nolimits}

\newcommand{\Gm}{\mathbb G}

\newcommand{\Z}{{\mathbb Z}}

\newcommand{\lrep}[1]{\mathop{#1\textrm{-}\mathrm{rep}}}

\newcommand{\Spec}{\operatorname{Spec}}
\newcommand{\Smk}{\mathcal{S}m_k}
\newcommand{\ksep}{k^{\mathrm{sep}}}
\newcommand{\SmGk}{\mathcal{S}m^G_k}
\newcommand{\op}{{\mathrm{op}}}
\newcommand{\PGL}{\mathrm{PGL}}
\newcommand{\Gal}{\mathrm{Gal}}
\newcommand{\mc}[1]{\mathcal{#1}}
\newcommand{\mr}[1]{\mathrm{#1}}

\newcommand{\refbr}[1]{{\hyperref[#1]{(\ref*{#1})}}}

\newcommand{\lebru}{\le_{\scriptscriptstyle \mathrm{B}}}
\newcommand{\leant}{\le_{\mathrm{a}}}

\newcommand{\lfaktor}[2]{
\mathchoice{
    \raisebox{-0.7pt}{$\displaystyle {#1}$}
    \mkern0.2mu \backslash \mkern0.2mu
    \raisebox{0.7pt}{$\displaystyle {#2}$}
    }
    {
    \raisebox{-0.7pt}{${#1}$}
    \mkern2.2mu \backslash \mkern1.2mu
    \raisebox{0.7pt}{${#2}$}
    }
    {
    \raisebox{-0.7pt}{$\scriptstyle {#1}$}
    \mkern-2.5mu \backslash \mkern0.5mu
    \raisebox{0.7pt}{$\scriptstyle {#2}$}
    }
    {
    \raisebox{-0.7pt}{$\scriptscriptstyle {#1}$}
    \mkern-3.2mu \backslash \mkern-0.2mu
    \raisebox{0.7pt}{$\scriptscriptstyle {#2}$}
    }
	}   

\newcommand{\rfaktor}[2]{
\mathchoice{
    \raisebox{0.7pt}{$\displaystyle {#1}$}
    \mkern-0.2mu / \mkern-0.2mu
    \raisebox{-0.7pt}{$\displaystyle {#2}$}
    }
    {
    \raisebox{0.7pt}{${#1}$}
    \mkern-0.2mu / \mkern-0.2mu
    \raisebox{-0.7pt}{${#2}$}
    }
    {
    \raisebox{0.5pt}{$\scriptstyle {#1}$}
    \mkern-2.5mu / \mkern-0.5mu
    \raisebox{-0.5pt}{$\scriptstyle {#2}$}
    }
    {
    \raisebox{0.1pt}{$\scriptscriptstyle {#1}$}
    \mkern-0.2mu / \mkern-0.2mu
    \raisebox{-0.1pt}{$\scriptscriptstyle {#2}$}
    }
	}

\def\text#1{\mbox{#1}}

\newtheorem{theorem}{Theorem}[subsection]
\newtheorem{theorem_intr}{Theorem}[section]
\newtheorem{corollary_intr}[theorem_intr]{Corollary}
\newtheorem{lemma}[theorem]{Lemma}
\newtheorem{proposition}[theorem]{Proposition}
\newtheorem{corollary}[theorem]{Corollary}

\theoremstyle{definition}
\newtheorem{definition}[theorem]{Definition}
\newtheorem{example}[theorem]{Example}

\theoremstyle{remark}
\newtheorem{remark}[theorem]{Remark}

\numberwithin{equation}{theorem}

\long\def\comment#1{}
\newcommand{\Addresses}{{
    \bigskip
	\footnotesize
		
	Alexey~Ananyevskiy, Mathematisches Institut, Ludwig-Maximilians-Universit\"at M\"unchen, Theresienstr. 39, D-80333 M\"unchen, Germany. \textit{E-mail address: alseang@gmail.com}
	\medskip
		
	Alexander~Samokhin, Fakult\"at f\"ur Mathematik, Universit\"at Bielefeld, 33501, Bielefeld, Germany.
    \textit{E-mail address: alexander.samokhin@math.uni-bielefeld.de}
}}

\begin{document}

\title{\bf Semiorthogonal decompositions\\ for families of twisted flag varieties}
	
\author{Alexey~Ananyevskiy and Alexander~Samokhin}

\date{}

\maketitle
	
\thanks{}
    
\begin{abstract}
	We show that the derived categories of smooth families of twisted generalized flag varieties admit semiorthogonal decompositions into derived categories of twisted sheaves over the base. In particular, we obtain a categorification of Panin’s computation of the Quillen K-theory of twisted flag varieties. The main ingredient is a generalization of the Samokhin–van der Kallen semiorthogonal decomposition for representation categories of parabolic subgroups from split simply connected semisimple groups to arbitrary quasi-split semisimple groups, including non-simply connected cases.
\end{abstract}

\tableofcontents

\section{Introduction}\label{sec:intro}

Derived categories of coherent sheaves on algebraic varieties were introduced in the work of Grothendieck and Verdier as a natural framework for studying derived functors and coherent duality. Although initially regarded primarily as a technical tool from homological algebra, they are now recognized as powerful and intrinsically interesting invariants of algebraic varieties. These invariants encode substantial geometric information, for example, when the canonical or anti-canonical sheaf is ample one can reconstruct the variety $X$ from its bounded derived category $\D^b(X)$ \cite[Theorem~2.5]{BO01}. It is therefore not surprising that these categories are, in general, difficult to understand.
Nevertheless, in special cases one can decompose derived categories into simpler, more manageable pieces, and such decompositions are often closely tied to the geometry of the underlying variety. The first example of such a decomposition into the simplest possible building blocks $\D^b(k)$ was given by Beilinson \cite{B79}:
\[
\D^b(\mathbb{P}^n_k) = \langle \D^b(k), \D^b(k),\hdots, \D^b(k)\rangle.
\]
Kapranov later generalized this result to flag varieties of type $A_n$ \cite{Kap84} and to quadrics \cite{Kap88}. These decompositions arise from full exceptional collections, that is, ordered sequences of objects in $\D^b(X)$ which generate the entire category, have minimal derived endomorphism algebras, and satisfy a semiorthogonality condition ensuring the absence of (shifted) morphisms in $\D^b(X)$ going backwards in the sequence. These results led to the conjecture that every projective variety $\rfaktor{G}{P}$ homogeneous under a split semisimple group G admits a full exceptional collection, and hence that its derived category decomposes into copies of $\D^b(k)$. This conjecture attracted considerable attention, and many special cases were established over the years; we refer the reader to the survey \cite{Fon24}. The recent paper \cite{SvdK24} settles the conjecture in full generality and in a uniform way. There, the authors construct a semiorthogonal decomposition 
\[
\D^b(\rep P) = \langle \D^b(\rep G), \D^b(\rep G),\hdots, \D^b(\rep G) \rangle
\]
for a split simply connected semisimple group $G$ and a parabolic subgroup $P\leqslant G$, and then transfer this decomposition to $\rfaktor{G}{P}$ using the standard associated sheaf functor $\D^b(\rep P)\to \D^b(\rfaktor{G}{P})$.

Once the conjecture is settled for varieties of the form $\rfaktor{G}{P}$, it is natural to ask what happens in families and for twisted forms of such varieties.
Several partial results have been known for some time. Orlov showed that the derived category of a relative flag bundle associated to a vector bundle decomposes into components equivalent to the derived category of the base \cite{O93}. Kuznetsov obtained decompositions for flat (not necessarily smooth) quadric fibrations \cite{Kuz08}. Bernardara proved that the derived category of a Severi--Brauer scheme $\mr{SB}(\mc{A})\to S$ decomposes into copies of derived categories of twisted sheaves over the base \cite{Ber09}:
\[
\D^b(\mr{SB}(\mc{A})) = \langle \D^b(S), \D^b(S,\mc{A}),\hdots, \D^b(S,\mc{A}^{\otimes n})\rangle.
\]
Similar results were obtained for generalized Severi--Brauer varieties and schemes \cite{Bl12, Ba16, DRC25}.
One can also ask a simpler question about the decomposition of the $\mr{K}$-motive $\mr{K}(X)$ of the twisted form of $\rfaktor{G}{P}$, and this was completely settled by Panin \cite{Pan94}, who showed that 
\[
\mr{K}(X) \cong \bigoplus \mr{K}(\mr{A}_i),
\]
where $\mr{A}_i$ are separable algebras introduced by Tits \cite{Tits71}.

In this paper, we generalize most of the above results by showing that for twisted forms of projective homogeneous varieties and for families arising from torsors, the bounded derived category of the total space admits a decomposition into components given by derived categories of twisted coherent sheaves on the base.
\begin{theorem_intr}[{Theorem~\ref{thm:twisted_flags}}] \label{thm:twisted_flags_intr}
    Let $G$ be a quasi-split semisimple group over a field $k$, $P\leqslant G$ be a parabolic subgroup and $\mathcal{E}\to S$ be a right $G$-torsor over a smooth variety $S$ over $k$. Then there exists an $S$-linear semiorthogonal decomposition 
    \[
    \D^b(\rfaktor{\mc{E}}{P})= \langle \D_1,\D_2,\hdots,\D_n \rangle
    \]
    with $\D_i\cong \D^b(S,\mc{A}_i)$ being the bounded derived category of coherent sheaves of modules over a sheaf of coherent separable algebras $\mc{A}_i$ over $S$. Furthermore, the following holds:
    \begin{enumerate}
        \item If $G$ is split and simply connected then there exists a semiorthogonal decomposition as above with $\D_i\cong \D^b(S)$ and $n=[W(G):W(P)]$.    
        \item If $G$ is simply connected then there exists a semiorthogonal decomposition as above with $\D_i\cong \D^b(S_{k_i})$ being the derived category of the base change of $S$ to a finite separable field extension $k_i/k$ and with $\sum_{i=1}^n [k_i:k]=[W(G):W(P)]$.
        \item If $G$ is split then all $\mc{A}_i$ may be chosen to be sheaves of Azumaya algebras and $n=[W(G):W(P)]$.
    \end{enumerate}
    Here $W(G)$ and $W(P)$ are the Weyl groups of $G$ and $P$ respectively.
\end{theorem_intr}

\noindent This decomposition arises from a full $S$-relative separable-exceptional collection in $\D^b(\rfaktor{\mc{E}}{P})$ (see Definition~\ref{def:relative_decomposition}) which is a mild generalization of the notion of a fiberwise full exceptional collection. The construction involves a choice, namely, a total order on $W(G)$, and different choices may yield additional orthogonality properties, see Remark~{\hyperref[rem:comments_on_SOD_twisted_flag]{\ref*{rem:comments_on_SOD_twisted_flag}.2}}. The following special cases recover the known results:
\begin{itemize}
    \item $G=\mr{SL}_n$: flag bundles of vector bundles with trivial determinant,
    \item $G=\mr{PGO}_n$, $P=P_1$: smooth quadric fibrations,
    \item $G=\mr{PGL}_n$: generalized Severi--Brauer schemes (including flag bundles of arbitrary vector bundles).
\end{itemize}
The sheaves of algebras appearing above are explicit and are closely related to Tits algebras \cite{Tits71}, see Remark~\ref{rem:comments_on_SOD_twisted_flag}. For $S=\Spec k$ we obtain the following categorification of Panin's decomposition of $\mr{K}$-motives. 
\begin{corollary_intr}[{Corollary~\ref{cor:Panin}}]
    Let $X\in\Smk$ be a smooth projective variety homogeneous under an action of a reductive group. Then there exists a semiorthogonal decomposition
    \[
        \D^b(X)= \langle \D_1,\D_2,\hdots,\D_n \rangle
    \]
    with $\D_i\cong \D^b(A_i)$ and $A_i$ being a separable algebra over $k$ (Tits algebra).
\end{corollary_intr}
\noindent See also Example~\ref{ex:unitary} for the treatment of a particular case of a unitary Grassmannian, which illustrates this semiorthogonal decomposition.

Theorem~\ref{thm:twisted_flags_intr} is obtained by twisting a universal decomposition of derived categories of representations given in the next theorem via a base-change argument (Theorem~\ref{thm:base_change}).

\begin{theorem_intr}[{Theorem~\ref{thm:SOD_general}}] \label{thm:rep_intr}
   Let $G$ be a quasi-split semisimple algebraic group over a field $k$ and $P\leqslant G$ be a parabolic subgroup. Then there is a semiorthogonal decomposition
    \[
        \D^b(\rep P) = \langle \D_1,\D_2,\hdots,\D_n \rangle
    \]
    with $\D_i \cong \D^b(\lrep{A_i} G)$. Here $A_i$ is the matrix algebra $\mr{M}_{n_i}(k_i)$ equipped with an action of $G$, $k_i/k$ is a finite separable field extension, and $\lrep{A_i} G$ is the abelian category of $A_i$-representations of $G$, that is the category of modules over $A_i$ in $\rep G$. This decomposition arises from a full $G$-relative separable-exceptional collection (see~Definition~\ref{def:rel_rep}) in $\D^b(\rep P)$.
\end{theorem_intr}
The reference to universality is justified by the fact that every $\rfaktor{\mc{E}}{P}\to S$ fits in a Cartesian square
\[
\xymatrix{
\rfaktor{\mc{E}}{P} \ar[r]\ar[d]_{\pi_{\mc{E}}} & \mc{B}P \ar[d]\\
S \ar[r]^\xi & \mc{B}G
}
\]
of stacks, with $\mc{B}P$ and $\mc{B}G$ being the respective classifying stacks and the bottom horizontal arrow $\xi$ classifying the torsor $\mc{E}$. Since $\D^b(\mc{B}P)=\D^b(\rep P)$ and $\D^b(\mc{B}G)=\D^b(\rep G)$, the decomposition of Theorem~\ref{thm:twisted_flags_intr} may be viewed as a pullback of the decomposition of Theorem~\ref{thm:rep_intr}. Although the K\"unneth formula for derived categories of stacks may be a tricky question (see \cite[Theorem~1.2 and Remark~1.3]{BZFN10} and note that the stacks $\mc{B}G$ and $\mc{B}P$ are not perfect in the positive characteristic \cite{HNR19}), the situation here is simpler since the relative anti-canonical line bundle $\omega^*_{\pi_{\mc{E}}}$ is ample (Lemma~\ref{lem:canonical_antiample}) and clearly comes as a pullback (Lemma~\ref{lem:canonical_equivariant}), which ensures the necessary generation. 

The proof of Theorem~\ref{thm:SOD_general} begins with \cite[Theorems~14.2 and~14.3]{SvdK24}, which treats the split simply connected case. In order to cover the quasi-split case we carefully adjust the constructions of \cite{SvdK24} taking into account the fields of definition of elements of the Weyl group and the respective Joseph and relative Schubert modules, along the way obtaining some additional orthogonality properties and certain independence from the choice of a total order (Proposition~\ref{prop:more_orthogonality}) for the exceptional collections constructed in \cite{SvdK24}. Then we pass from the simply connected case to a general one adapting the ideas from \cite{Ela09} to the setting of groups of multiplicative type.

The exposition in the paper is slightly non-linear with the main result of Section~\ref{sec:SODs_twisted_flags} making use of the main result of Section~\ref{sec:relative_split}. The reason is that the first half of the paper, including Section~\ref{sec:SODs_twisted_flags}, uses the language of equivariant sheaves, while the second half of the text makes use of the representation-theoretic terminology. If the reader prefers a linear exposition, we recommend first skipping Section~\ref{sec:SODs_twisted_flags} and returning to it after Section 3. 

\textbf{Acknowledgments.} The work of the first author is supported by DFG Heisenberg grant AN 1545/1-1 and by the DFG research grant AN 1545/4-1. The work of the second author is supported in part by the Deutsche Forschungsgemeinschaft (SFB-TRR 358/1 2023 - 491392403). He also acknowledges the hospitality of the LMU M\"unchen during his visits in May 2025 and in July 2026.
The authors would like to thank Andrei Lavrenov, Vladimir Sosnilo and Wilberd van der Kallen for valuable discussions. The first author would also like to thank Ivan Panin, whose influence on this project cannot be overstated. In particular, it was through the work on the thesis written under Ivan Panin’s direction (\cite{Ana12} was a part of the project) that his interest in this subject first developed.

\textbf{AI declaration.} No AI tools were used in the mathematical research or writing of the present~paper.

Throughout the paper, we employ the following assumptions, conventions, and notations.
\begin{itemize} \itemsep0pt
	\item A \textit{variety} over a field $k$ is a separated reduced scheme of finite type over $k$.
    \item We mostly follow \cite{Jan03} and \cite{Mi17} in the treatment of algebraic groups. An algebraic group over a field $k$ is a smooth algebraic group variety over $k$, and a parabolic subgroup $P\leqslant G$ of an affine algebraic group $G$ over a field $k$ is a smooth algebraic subgroup such that the variety $\rfaktor{G}{P}$ is proper over $k$ \cite[Definition~17.15, Remark~17.30]{Mi17}.
    \item By a full subcategory, we mean a strictly full subcategory.
    \item All functors between derived categories are implicitly assumed to be derived.
    \item We often omit pullback, restriction of scalars, and forgetful functors from the notation. If a tensor product is taken over the structure sheaf (resp. base field) we omit the structure sheaf (resp. the base field) from the subscript.

\end{itemize}		 

\begin{tabular}{l|l}
	$k$ & a field \\
	$\Smk$ & the category of smooth varieties over $k$\\
    $G$ & an affine algebraic group over $k$\\
    $\SmGk$ & the category of smooth varieties over $k$ with a left $G$-action\\
    $\Coh_G X$, $\QCoh_G X$ & the categories of equivariant coherent and quasi-coherent sheaves on $X$\\
    $\CohAlg_G X$, $\SAlg_G X$ & the categories of coherent sheaves of algebras and separable algebras on $X$\\
    $\D^b(X,\mc{A})$ & the bounded derived category of left $\mc{A}$-modules for $\mc{A}\in \SAlg_G X$\\
    $\iHom$, $\iEnd$ & the internal Hom and End functors\\
    $\hull(\mc{S})$ & the smallest triangulated subcategory containing $\mc{S}$\\
    $\lrep{K} G$ & the category of finite dimensional $K$-representations of $G$\\
    $\SAlg_G k$ & the category of $G$-equivariant separable algebras over $k$\\
    $\CSAlg_G k$ & the category of $G$-equivariant central simple algebras over $k$\\
    $\D^b(\lrep{A} G)$ & the bounded derived category of left $A$-modules for $A\in \SAlg_G k$\\
\end{tabular}

\section{Relative SODs for derived categories of equivariant sheaves}

\subsection{Recollection on equivariant derived categories of twisted sheaves}
\label{sec:recollection_derived}

In this section we recall some generalities on derived categories of equivariant twisted coherent sheaves and introduce the relevant notation. Most of the setup could also be packed using the language of stacks, but we deliberately choose to stay in a more elementary setting to avoid some unpleasant technicalities. By the same reason we deal only with the bounded derived categories of smooth varieties, although one can also treat the singular case using the categories of perfect complexes.

In this section $G$ is an affine algebraic group over a field $k$.

\begin{definition} \label{def:derived_Azumaya}
For $X\in \SmGk$ we denote $\Coh_G X$ and $\QCoh_G X$ the abelian categories of coherent and quasi-coherent $G$-equivariant sheaves on $X$ as in \cite[\S~1.2]{Th87}, and by $\CohAlg_G X$ the category of monoids in $\Coh_G X$, that is, the category of \textit{equivariant coherent sheaves of algebras} on $X$. An algebra $\mc{A}\in \CohAlg_G X$ gives rise to the abelian categories $\Coh_{G} (X,\mc{A})$ and $\QCoh_{G} (X,\mc{A})$ of coherent and quasi-coherent $G$-equivariant left $\mc{A}$-modules. We denote $\D(\QCoh_{G} (X,\mc{A}))$ the derived category of the latter one, and $\D^b_G(X,\mc{A}) \subseteq \D(\QCoh_{G} (X,\mc{A}))$ its full triangulated subcategory consisting of complexes with cohomology sheaves being coherent and concentrated in a bounded interval. A standard argument shows that $\D^b_G(X,\mc{A})$ is equivalent to the bounded derived category $\D^b(\Coh_{G} (X,\mc{A}))$.
If either $\mathcal{A}$ or $G$ or both are trivial, then we drop them from the notation obtaining respectively
\begin{itemize}
    \item
    the bounded derived category of equivariant sheaves $\D^b_G(X)$, see \cite[Appendix~A]{MR16} for some basic properties,
    \item
    the bounded derived category $\D^b(X,\mc{A})$ of the non-commutative scheme $(X,\mc{A})$ \cite[\S~3]{BDG17}, which is especially well-behaved in the case of $\mc{A}$ being a sheaf of separable algebras \cite[Appendix~D]{Kuz06}, in particular, admitting a description via twisted sheaves when $\mc{A}$ is a sheaf of Azumaya algebras \cite[Theorem~1.3.7]{Cal00},
    \item 
    the classical bounded derived category of coherent sheaves $\D^b(X)$, see e.g. \cite[Chapter~III]{Huy06}).
\end{itemize}
The forgetful functor $\QCoh_{G} (X,\mc{A}) \to \QCoh (X)$ is exact and gives rise to exact forgetful functors 
\[
\Coh_{G} (X,\mc{A}) \to \Coh (X),\quad \D(\QCoh_{G} (X,\mc{A}))\to \D(\QCoh (X)),\quad \D^b_G(X,\mc{A}) \to \D^b(X),
\]
For $X,Y\in \SmGk$, a $G$-equivariant morphism $f\colon Y\to X$ and $\mc{A}\in \CohAlg_G X$ we put 
\[
\Coh_{G} (Y,\mc{A}):=\Coh_{G} (Y,f^*\mc{A}),\quad \D^b_{G}(Y,\mc{A}):=\D^b_{G}(Y,f^*\mc{A}),
\]
and similarly for the quasi-coherent sheaves and its derived categories. In the paper we are interested only in the case of $\mc{A}\in\CohAlg_G X$ being an equivariant sheaf of \textit{separable algebras} in the sense of the following definition.
\end{definition}

\begin{definition}
    Let $R$ be a commutative ring and $A$ be an $R$-algebra. We say that $A$ is \textit{separable over $R$} (see \cite[\S~1]{AG60} or \cite[Chapter~III]{KO74}) if the surjection
    \[
    A\otimes_R A^{\op}\to A,\quad a\otimes b\mapsto ab,
    \]
    splits in the category of left $A\otimes_R A^{\op}$-modules, or, equivalently, if $A$ is a projective left $A\otimes_R A^{\op}$-module. A separable over $R$ algebra $A$ is an \textit{Azumaya algebra over $R$} if additionally $R$ coincides with the center of $A$. Alternatively, an $R$-algebra $A$ is an Azumaya algebra over $R$ if it is finitely generated and projective as an $R$-module and the homomorphism
    \[
    A\otimes_R A^{\op} \to \operatorname{End}_R A,\quad a\otimes b\mapsto (x\mapsto axb),
    \]
    is an isomorphism. See also \cite[Th\'eor\`eme~III.5.1]{KO74} for some other equivalent conditions. Over a field an Azumaya algebra is the same as a \textit{central simple algebra}, and a separable algebra is a finite product of central simple algebras over finite separable field extensions of the base field.
    
    For $X\in\Smk$ and $\mathcal{A}\in \CohAlg X$ we say that $\mathcal{A}$ is a \textit{sheaf of separable} (resp. \textit{Azumaya}) \textit{algebras} if for every $x\in X$ the stalk $\mc{A}_x$ is a separable (resp. Azumaya) algebra over $\mc{O}_{X,x}$. Alternatively, a coherent sheaf of algebras is a sheaf of 
    \begin{itemize}
        \item Azumaya algebras if it is isomorphic locally in \'etale topology to a sheaf of matrix algebras $M_{n}(\mc{O}_X)$, see e.g. \cite[Th\'eor\`eme~III.6.4]{KO74} or \cite[Proposition~IV.2.1]{Mi80},
        \item separable algebras if it is isomorphic locally in \'etale topology to a product of sheaves of matrix algebras $\prod_{i=1}^{m} M_{n_i}(\mc{O}_X)$, this follows from the Azumaya case combined with \cite[Theorem~2.3]{AG60} and \cite[Th\'eor\`eme~III.4.7]{KO74}.
    \end{itemize}
    For $X\in \SmGk$ we say that $\mc{A}\in \operatorname{CohAlg}_G X$ is an \textit{equivariant sheaf of separable} (resp. \textit{Azumaya}) \textit{algebras} if it becomes a sheaf of separable (resp. Azumaya) algebras after forgetting the equivariant structure. We denote $\SAlg_G X$ the full subcategory of $\CohAlg_G X$ consisting of equivariant sheaves of separable algebras.
\end{definition}

\begin{example}
    \begin{enumerate}
        \item An algebra $A$ over a field $k$ is separable if and only if $A\cong \prod_{i=1}^{m} M_{n_i}(D_i)$ with $D_i$ being finite dimensional division algebras over $k$ with the center $L_i$ being a separable field extension of $k$ \cite[Th\'eor\`eme~III.3.1]{KO74}. Similarly, an algebra $A$ over a field $k$ is an Azumaya algebra if and only if $A\cong M_{n}(D)$ with $D$ a finite dimensional central division algebra over $k$. 
        \item Let $X\in\SmGk$, then equivariant sheaves of Azumaya algebras $\mc{A},\mc{B}\in\SAlg_G X$ are called \textit{Brauer equivalent} if there exist equivariant vector bundles $\mc{V}_1,\mc{V}_2$ over $X$ such that $\mc{A}\otimes \underline{\operatorname{End}}\,\mc{V}_1\cong \mc{B}\otimes \underline{\operatorname{End}}\,\mc{V}_2$ as equivariant sheaves of algebras. The equivalence classes of equivariant sheaves of Azumaya algebras form the Brauer group $Br_G(X)$ with the group operation given by tensor product (see \cite[\S~1.2]{Gro68} for the non-equivariant case).
        \item The matrix algebra $M_n(k)$ equipped with the adjoint action of $\PGL_n$ is a $\PGL_n$-equivariant Azumaya algebra over $\Spec k$. Note that for $n\ge 2$ the algebra $M_n(k)$ is not an endomorphism algebra of a $\PGL_n$-representation, so it represents a non-trivial element in the $\PGL_n$-equivariant Brauer group of $\Spec k$. Moreover, there are no \'etale covers of $\Spec k$ that make $M_n(k)$ an endomorphism algebra of a $\PGL_n$-representation, so equivariant Azumaya algebras may be \'etale-locally non-trivial.
        \item Let $G$ be a split semisimple group over $k$, let $\tilde{G}$ be its simply connected cover and $M$ be an irreducible representation of $\tilde{G}$. Then the conjugation action of $\tilde{G}$ on $\End_k M$ descends to an action of $G$ giving rise to an Azumaya algebra $A_M \in \SAlg_{G} \Spec k$ such that $\res^G_{\tilde{G}} A_M\cong \iEnd M$, see Definition~\ref{def:Azumaya_equiv} for more details. These are closely related to the \textit{Tits algebras} introduced in \cite[\S~4]{Tits71}, see the discussion in Definition~\ref{def:equiv_Tits} below.
        \item Let $k=\mathbb{R}$ and $\mathrm{PGU}_n$ be the projective unitary group viewed as an algebraic group over $\mathbb{R}$. The matrix algebra $M_n(\mathbb{C})$ equipped with the conjugation action of $\mathrm{PGU}_n$ is a $\mathrm{PGU}_n$-equivariant separable algebra over $\mathbb{R}$ which is not Azumaya. Similar examples arise for outer forms of adjoint split semisimple groups, in particular, for adjoint quasi-split non-split semisimple groups.
    \end{enumerate}
\end{example}

For $X\in\SmGk$ and $\mc{A}\in \SAlg_G X$ the categories $\D^b_G(X,\mc{A})$ come with the usual package of $\otimes$, $\underline{\Hom}$, $f^*$ and $f_*$, projection formula and base change which we sum up in the proposition below.

\begin{proposition} \label{prop:functor_formalism}
\begin{enumerate}
    \item 
    Let $X,Y\in \SmGk$, $f\colon Y\to X$ be a $G$-equivariant morphism and $\mc{A}\in\SAlg_G X$. Then there is a pullback functor 
    \[
    f^*\colon \D^b_G(X,\mc{A}) \to \D^b_G(Y,\mc{A}).
    \]
    If $f$ is proper then there is also a direct image functor 
    \[
    f_*\colon \D^b_G(Y,\mc{A}) \to \D^b_G(X,\mc{A}),
    \]
    which is right adjoint to $f^*$. Forgetting the $G$-equivariant and $\mc{A}$-module structure these functors realise to the usual $f^*$ and $f_*$.
    \item (Base change). For a cartesian diagram
    \[
        \xymatrix{
        X\times_Z Y \ar[r]^(0.6){g'} \ar[d]_{f'} & X \ar[d]^f\\
        Y\ar[r]^g & Z 
        }
        \]
    in $\SmGk$ with $g$ being flat and $f$ being proper, and for $\mc{A}\in \SAlg_G Z$, one has a canonical isomorphism 
    \[  
        g^*\circ f_*\cong f'_* \circ (g')^* \colon \D^b_G(X,\mc{A})\to \D^b_G(Y,\mc{A}).
    \]
    
    \item 
    For $X\in \SmGk$ and $\mc{A}_1,\mc{A}_2,\mc{A}_3\in\SAlg_G X$ there is an associative bilinear pairing
    \[
    \D^b_G(X,\mc{A}_1\otimes \mc{A}_2^{\op}) \times \D^b_G(X,\mc{A}_2\otimes \mc{A}_3)\xrightarrow{-\otimes_{\mc{A}_2}-} \D^b_G(X,\mc{A}_1\otimes \mc{A}_3).
    \]
    Pullback functors are monoidal with respect to this pairing. For $N\in \D^b_G(X,\mc{A}_2\otimes \mc{A}_3)$ the functor $-\otimes_{\mc{A}_2}N$ admits a right adjoint
    \[
    \iHom_{\D^b_G(X,\mc{A}_3)}(N,-)\colon \D^b_G(X,\mc{A}_1\otimes \mc{A}_3)\to \D^b_G(X,\mc{A}_1\otimes \mc{A}_2^\op).
    \]
    Furthermore, for
    \[
    N^*:=\underline{\Hom}_{\D^b_G(X)}(N,\mc{O}_X) \in \D^b_G(X,\mc{A}_2^{\op}\otimes \mc{A}_3^{\op})
    \]
    and $L\in \D^b(X,\mc{A}_1\otimes\mc{A}_3)$ the natural morphism $N^*\otimes_{\mc{A}_3} L\to \iHom_{\D^b_G(X,\mc{A}_3)}(N,L)$ is an isomorphism.

    \item (Projection formula). For $X,Y\in \SmGk$, $\mc{A}_1,\mc{A}_2,\mc{A}_3\in\SAlg_G X$, a proper $G$-equivariant morphism $f\colon Y\to X$, and $M\in \D^b_G(X,\mc{A}_1\otimes \mc{A}_2^{\op})$, $N\in \D^b_G(Y,\mc{A}_2\otimes \mc{A}_3)$ there is an isomorphism
    \[
    M\otimes_{\mc{A}_2} f_*N\cong f_*(f^*M\otimes_{\mc{A}_2} N).
    \]
\end{enumerate}
\end{proposition}

\begin{proof}
    This is standard. The passage from the classical non-equivariant setting to the equivariant one is discussed in \cite[Appendix~A]{MR16} and the passage from the classical setting to $\mc{A}$-modules is discussed in \cite[Appendix~D]{Kuz06} (note that although the section is called \textit{Azumaya algebraic varieties}, the author actually treats coherent sheaves of separable algebras). The general case follows by combining these two approaches.
\end{proof}

\subsection{Relative equivariantly exceptional collections and SODs}
\label{sec:rel_ex}
In this section we recall the notion of a semiorthogonal decomposition of a triangulated category and introduce relative equivariantly separable-exceptional collections, which are a mild generalization of fiberwise exceptional collections and give rise to semiorthogonal decompositions in families.

In this section $G$ is an affine algebraic group over $k$.

\begin{definition} \label{def:thetaX}
    
        
    Let $\mc{W},S\in\SmGk$, $\mc{A}\in \SAlg_G S$ and $f\colon \mc{W}\to S$ be a proper $G$-equivariant morphism. For $X\in \D^b_G(\mc{W},\mc{A})$ we denote $\theta_X\colon \mc{A}\to f_*\iEnd_{\D^b_{G}(\mc{W})} X$ the structure morphism
    corresponding to the identity morphism under the isomorphisms
    \begin{equation}
    \begin{multlined}
    \Hom_{\D^b_G(\mc{W},\mc{A})}(X,X) \cong \Hom_{\D^b_G(\mc{W},\mc{A})}(f^*\mc{A}\otimes_{f^*\mc{A}} X,X) \cong \\ \cong \Hom_{\D^b_G(\mc{W},\mc{A}\otimes \mc{A}^\op)}(f^*\mc{A},\iEnd_{\D^b_{G}(\mc{W})} X) 
    \cong \Hom_{\D^b_G(\mc{S},\mc{A}\otimes \mc{A}^\op)}(\mc{A},f_*\iEnd_{\D^b_{G}(\mc{W})} X).
    \end{multlined}
    \end{equation}
\end{definition}
    
\begin{definition} \label{def:relative_decomposition}    
    Let $\mc{W},S\in\SmGk$, $\mc{A}\in \SAlg_G S$ and $f\colon \mc{W}\to S$ be a proper $G$-equivariant morphism. We say that $X\in \D^b_G(\mc{W},\mc{A})$ is \textit{$S$-relative equivariantly $\mc{A}$-exceptional} (or \textit{$S$-relative equivariantly separable-exceptional}) if the structure morphism 
    \[
    \theta_X \colon \mc{A}\to f_*\iEnd_{\D^b_{G}(\mc{W})} X
    \]
    of Definition~\ref{def:thetaX} is an isomorphism. Furthermore, $X$ is called
    \begin{itemize}
        \item \textit{$S$-relative equivariantly Azumaya-exceptional} if $\mc{A}$ is a $G$-equivariant sheaf of Azumaya algebras, 
        \item \textit{$S$-relative equivariantly \'etale-exceptional} if $\mc{A}$ is a $G$-equivariant sheaf of commutative separable algebras,
        \item \textit{$S$-relative equivariantly exceptional} if $\mc{A}=\mc{O}_S$.
    \end{itemize}
    \noindent
    A sequence of objects $X_i\in \D^b_G(\mc{W},\mc{A}_i)$, $1\le i\le n$, is called an \textit{$S$-relative equivariantly separable-exceptional (resp. Azumaya-exceptional, resp. \'etale-exceptional, resp. exceptional) collection in $\D^b_G(\mc{W})$} if 
    \begin{enumerate}
        \item $X_i$ is $S$-relative equivariantly separable-exceptional (resp. Azumaya-exceptional, resp. \'etale-exceptio\-nal, resp. exceptional) for every $1\le i\le n$,
        \item $f_* \iHom_{\D^b_G(\mc{W})}(X_i,X_j) =0$ for every $1\le j<i\le n$.
    \end{enumerate}
    \noindent
    An $S$-relative equivariantly separable-exceptional collection $X_i\in \D^b_G(\mc{W},\mc{A}_i)$, $1\le i\le n$, is \textit{full} if 
    \[
    \D^b_G(\mc{W}) = \hull \left( \D^b_G(S,\mc{A}_1^\op) \otimes_{\mc{A}_1} X_1,\hdots, \D^b_G(S,\mc{A}_n^\op) \otimes_{\mc{A}_n} X_n\right),
    \]
    where on the right stands the smallest triangulated subcategory of $\D^b_G(\mc{W})$ containing $M_i\otimes_{\mc{A}_i} X_i$ for all $M_i\in \D^b_G(S,\mc{A}_i^\op)$, $1\le i \le n$.

    If $S=\Spec k$ (resp. $G$ is a trivial group) then a (full) $S$-relative equivariantly separable-exceptional collection is called a (full) equivariantly (resp. $S$-relative) separable-exceptional collection, and similarly for Azumaya-exceptional, \'etale-exceptional and exceptional collections.    
\end{definition}

\begin{remark}
    Let $G$ be the trivial group and $S=\Spec k$, in this case the above definition of a (full) $S$-relative equivariantly exceptional collection coincides with the usual notion of a (full) exceptional collection (see e.g. \cite[Definition~1.57]{Huy06}), while a separable-exceptional collection is a variant of the notion of semi-exceptional collection of \cite[Definition~1.9]{O20} (precisely, we ask that the semisimple algebras of loc. cit. are in fact separable) and also is a mild generalization of \cite[Definitions~3.2 and~3.3]{Nov23}, with the difference being that we allow arbitrary separable algebras rather than division algebras as in loc. cit. In the non-equivariant case the notion of a relative exceptional collection was also explicitly given in \cite[Definition~3.19]{BLMNPS21}, but was already implicitly used in \cite[Theorem~3.1]{Sam07}, where it is introduced in an equivalent form of a fiberwise exceptional collection.
\end{remark}


    \begin{definition}[{\cite[\S~1]{BK89},\cite[\S~1.4]{Huy06}}]
        Let $\mc{D}$ be a triangulated category and $\mc{D}'\subseteq \D$ be a full triangulated subcategory. We say that $\mc{D}'$ is \textit{right admissible} (resp. \textit{left admissible}) if the inclusion $i_*\colon \mc{D}'\to \D$ admits a right adjoint $i^!\colon \mc{D}\to \D'$ (resp. a left adjoint $i^*\colon \mc{D}\to \D'$). A subcategory that is both right and left admissible is \textit{admissible}. We denote $(\D')^\perp$ the full subcategory of $\D$ consisting of those objects $A$ such that $\Hom_\D(B,A)=0$ for all $B\in \D'$. Similarly, ${}^\perp(\D')$ is the full subcategory of $\D$ consisting of those objects $A$ such that $\Hom_\D(A,B)=0$ for all $B\in \D'$.
    
        Let $\mc{D}$ be a triangulated category and $\D_1,\D_2\subseteq \mc{D}$ be full triangulated subcategories. We say that
        \[
            \mc{D}= \langle \D_1,\D_2\rangle
        \]
        is a \textit{semiorthogonal decomposition} if 
        \begin{enumerate}
            \item for all $M_1\in \D_1$, $M_2\in \D_2$, one has $\Hom(M_2,M_1)=0$,
            \item $\D=\hull (\D_1,\D_2)$.
        \end{enumerate}
        Note that in the presence of the condition (1) the condition (2) is equivalent to the following: for every $M\in \D$ there exists a distinguished triangle
        \[
        M_2 \to M\to M_1 \to M_2[1]
        \]
        with $M_1\in \D_1$ and $M_2\in \D_2$. Moreover, this triangle is essentially unique and functorial in $M$. It follows that $\mc{D}_1$ is left admissible with the left adjoint to the inclusion $j_*\colon \D_1\to \D$ given by $j^*(M):=M_1$ and $\D_2$ is right admissible with the right adjoint to the inclusion $i_*\colon \D_2\to \D$ given by $i^!(M):=M_2$. Furthermore, it also follows that $\D_1=(\D_2)^\perp$ and $\D_2={}^\perp(\D_1)$, and one can show that every left (resp. right) admissible subcategory $\D'\subseteq \D$ defines a semiorthogonal decomposition $\D=\langle \D', {}^\perp(\D')\rangle$ (resp. $\D=\langle (\D')^\perp ,\D'\rangle$).
        
        Inductively, for a sequence $\mc{D}_1,\mc{D}_2,\hdots, \mc{D}_n \subseteq \mc{D}$ of full triangulated subcategories we say that
        \[
        \mc{D}= \langle \mc{D}_1,\mc{D}_2,\hdots,\mc{D}_n\rangle
        \]
        is a \textit{semiorthogonal decomposition} if
        \begin{enumerate}
            \item for all $i>j$ and $M_i\in \mc{D}_i$, $M_j\in \mc{D}_j$, one has $\Hom(M_i,M_j)=0$,
            \item $\D=\hull (\D_1,\D_2,\hdots,\D_n)$.
        \end{enumerate}
        Again, in the presence of the condition (1) the condition (2) is equivalent to the following: for every $M\in \D$ there exists a filtration
        \[
        0=M_{\ge n+1} \xrightarrow{f_n} M_{\ge n} \xrightarrow{f_{n-1}}\hdots \xrightarrow{f_{2}} M_{\ge 2}\xrightarrow{f_1} M_{\ge 1}=M
        \]
        such that $Cone(f_i)\in \D_i$ for every $1\le i\le n$. This filtration is essentially unique and functorial, and it follows that $\mc{D}_i\subseteq \hull\left(\mc{D}_1,\D_2\hdots, \D_i \right)$ is right admissible for all $i$.
    \end{definition}

    \begin{remark}
        There is some discrepancy regarding the admissibility condition for semiorthogonal decompositions between \cite[Definition~1.59]{Huy06} and \cite{BK89}, \cite[Definition~3.1]{BO02} and \cite[\S~1.1]{Kuz14}. The possible reason is that the subcategories appearing in examples of semiorthogonal decompositions of geometric origin are usually admissible (and saturated) by \cite[Proposition~2.6, Theorem~2.14]{BK89} and~\cite[Theorem~3.1.5]{BVdB03}, and all left (or right) admissible subcategories in $\D^b(X)$ for a smooth projective $X$ are automatically admissible (see e.g. \cite[\S~2.2]{Kuz11}).
    \end{remark}

    \begin{lemma} \label{lem:SOD_direct_sum}
        Let $\D=\langle \D_1,\D_2\rangle$ be a semiorthogonal decomposition with 
        \[
            j^*\colon \D\leftrightarrows \D_1\colon j_*,\quad i_*\colon \D_2\leftrightarrows \D\colon i^!
        \]
        being the corresponding adjunctions. Suppose that $\D=\C_1\oplus \C_2$. Then for $l=1,2$ one has
        \[
        j_*j^*(\C_l)\subseteq \C_l,\quad i_*i^!(\C_l)\subseteq \C_l.
        \]
    \end{lemma}
    \begin{proof}
        Straightforward.
    \end{proof}

    \begin{definition}[{\cite[\S~2.4]{Kuz14}}]
        Let $\mc{W},S\in\SmGk$, $f\colon \mc{W}\to S$ be a proper equivariant morphism. We say that a full triangulated subcategory $\mc{D}\subseteq \D^b_G(\mc{W})$ is \textit{$S$-linear} if for every $X\in \D$ and $M\in \D^b_G(\mc{S})$ one has $(f^*M)\otimes X\in \mc{D}$. A semiorthogonal decomposition $\D^b_G(\mc{W})=\langle \D_1,\D_2,\hdots,\D_n\rangle$ is \textit{$S$-linear} if all $\D_i$ are $S$-linear subcategories.
    \end{definition}        

    \begin{proposition} \label{prop:SOD_from_collection}
        Let $\mc{W},S\in\SmGk$, $f\colon \mc{W}\to S$ be a proper equivariant morphism and $X_i\in \D^b_G(\mc{W},\mc{A}_i)$, $1\le i\le n$, be an $S$-relative equivariantly separable-exceptional collection. Consider the functors
        \[
        \Phi_i\colon \D^b_G(S,\mc{A}_i^{\op}) \to \D^b_G(\mc{W}), \quad M\mapsto f^*M\otimes_{f^*\mc{A}_i} X_i.
        \]
        Then the following holds.
        \begin{enumerate}
            \item The functors $\Phi_i$ are fully faithful.
            \item
            For all $i>j$ and $M_i\in \Phi_i(\D^b_G(S,A_i^{\op}))$, $M_j\in \Phi_j(\D^b_G(S,A_j^{\op}))$ one has 
            \[
            \Hom_{\D^b_G(\mc{W})}(M_i,M_j)=0.
            \]
            \item The subcategories $\Phi_i(\D^b_G(S,\mc{A}_i^{\op}))\subseteq \D^b_G(\mc{W})$ are $S$-linear and admissible.
            \item The subcategory
            \[
            \hull\left( \Phi_1(\D^b_G(S,\mc{A}_1^{\op})),\Phi_2(\D^b_G(S,\mc{A}_2^{\op})),\hdots, \Phi_n(\D^b_G(S,\mc{A}_n^{\op})) \right) \subseteq \D^b_G(\mc{W})
            \]
            is $S$-linear and admissible. If the collection is full then there is an $S$-linear semiorthogonal decomposition
            \[
            \D^b_G(\mc{W})=\langle \Phi_1(\D^b_G(S,\mc{A}_1^{\op})),\Phi_2(\D^b_G(S,\mc{A}_2^{\op})),\hdots, \Phi_n(\D^b_G(S,\mc{A}_n^{\op}))\rangle.
            \]
        \end{enumerate}
    \end{proposition}
    \begin{proof}
        The reasoning is an adaptation of the standard one (see e.g. \cite[{\S~1.4}]{Huy06}), but we provide it for the sake of completeness.
    
        Below we mostly omit $f^*$ from the notation. Let $M\in \D^b_G(S,\mc{A}_i^{\op})$, $N\in \D^b_G(S,\mc{A}_j^{\op})$, then
        \begin{multline*}
        \Hom_{\D^b_G(\mc{W})} (M\otimes_{\mc{A}_i} X_i,N\otimes_{\mc{A}_j} X_j) 
        \cong \Hom_{\D^b_G(\mc{W},\mc{A}_i^\op\otimes \mc{A}_j)} ( f^*(M\otimes N^*) ,X_j\otimes X_i^*) \cong \\
        \cong \Hom_{\D^b_G(S,\mc{A}_i^\op\otimes \mc{A}_j)} (M\otimes N^* , f_*(X_j\otimes X_i^*)) 
        \cong \Hom_{\D^b_G(S,\mc{A}_i^\op\otimes \mc{A}_j)} (M\otimes N^* , f_*\iHom_{\D^b_G(\mc{W})}(X_i,X_j)).
        \end{multline*}
        
        1. If $i=j$ then we further have
        \begin{multline*}
        \Hom_{\D^b_G(S,\mc{A}_i^\op\otimes \mc{A}_i)} (M\otimes N^* , f_*\iHom_{\D^b_G(\mc{W})}(X_i,X_i))= \Hom_{\D^b_G(S,\mc{A}_i^\op\otimes \mc{A}_i)} (M\otimes N^* , \mc{A}_i)
        \cong \\
        \cong\Hom_{\D^b_G(S,\mc{A}_i^\op)} (M , N),
        \end{multline*}
        hence the functor $\Phi_i$ is fully faithful.
        
        2. If $i>j$ then $f_*\iHom_{\D^b_G(\mc{W})}(X_i,X_j)=0$, hence
        \[
        \Hom_{\D^b_G(S,\mc{A}_i^\op\otimes \mc{A}_j)} (M\otimes N^* , f_*\iHom_{\D^b_G(\mc{W})}(X_i,X_j)) =0
        \]
        and we obtain the semiorthogonality property.

        3. Let $M\in \D^b_G(S,\mc{A}_i^{\op})$ and $N\in \D^b_G(S)$, then
        \[
        N\otimes (M\otimes_{\mc{A}_i} X_i) \cong (N\otimes M)\otimes_{ A_i} X_i,
        \]
        hence $\Phi_i (\D^b_G(S,\mc{A}_i^{\op}))$ is an $S$-linear subcategory of $\D^b_G(\mc{W})$.
        
        Recall that by \cite[Remark~1.43.iii)]{Huy06} right admissibility of the subcategory $\Phi_i (\D^b_G(S,\mc{A}_i^{\op}))\subseteq \D^b_G(\mc{W})$ is equivalent to the property that for every $K\in \D^b_G(\mc{W})$ there exists $M\in \D^b_G(S,\mc{A}_i^\op)$ and a distinguished triangle
        \[
        M\otimes_{\mc{A}_i} X_i \to K \to L \to M\otimes_{\mc{A}_i} X_i[1]
        \]
        such that for every $N\in \D^b_G(S,\mc{A}_i^\op)$ one has $\Hom_{\D^b_G(\mc{W})}(N\otimes_{\mc{A}_i} X_i, L)=0$. This can be shown in a way similar to the proof of \cite[Lemma~1.58]{Huy06} as follows. Put
        \[
        M:=f_*\iHom_{\D^b_G(\mc{W})}(X_i,K) \in \D^b_G(S,\mc{A}_i^\op),
        \]
        and consider a distinguished triangle
        \[
        M\otimes_{\mc{A}_i} X_i \xrightarrow{\phi} K \to L \to M\otimes_{\mc{A}_i} X_i[1]
        \]
        with the first morphism
        \[
        M\otimes_{\mc{A}_i} X_i = f^*f_*\iHom_{\D^b_G(\mc{W})}(X_i,K) \otimes_{\mc{A}_i} X_i\xrightarrow{\phi} K
        \]
        arising from the counits of the adjunctions. Then for $N\in \D^b_G(S,\mc{A}_i^\op)$ using the first claim of the Proposition and adjunctions we obtain
        \begin{multline*}
        \Hom_{\D^b_G(\mc{W})}(N\otimes_{\mc{A}_i} X_i, M\otimes_{\mc{A}_i} X_i)
        \cong\Hom_{\D^b_G(S,\mc{A}_i^{\op})}(N , M)
        =\Hom_{\D^b_G(S,\mc{A}_i^{\op})}(N , f_*\iHom(X_i,K))
        \cong\\ 
        \cong \Hom_{\D^b_G(\mc{W},\mc{A}_i^{\op})}(f^*N , \iHom(X_i,K)) 
        \cong\Hom_{\D^b_G(\mc{W})}(N\otimes_{\mc{A}_i} X_i, K).
        \end{multline*}
        One can check that this isomorphism is given by composition with $\phi$. Hence $\Hom_{\D^b_G(\mc{W})}(N\otimes_{\mc{A}_i} X_i, L)=0$ and thus $\Phi_i (\D^b_G(S,\mc{A}_i^{\op}))\subseteq \D^b_G(\mc{W})$ is right admissible. A dual argument for $M:=(f_*\iHom_{\D^b_G(\mc{W})}(K,X_i))^*$ shows that $\Phi_i (\D^b_G(S,\mc{A}_i^{\op}))\subseteq \D^b_G(\mc{W})$ is left admissible.
    
        4. The triangulated hull of $S$-linear subcategories is clearly $S$-linear, and the triangulated hull of admissible semiorthogonal subcategories is admissible \cite[Proposition~1.12]{BK89}. The generation property of the semiorthogonal decomposition is precisely the fullness property of the collection.
    \end{proof}

    The following lemma is well-known.
    \begin{lemma} \label{lem:generetor}
         Let $\mc{W},S\in\Smk$, $f\colon \mc{W}\to S$ be a smooth proper morphism and $\D\subseteq \D^b(\mc{W})$ be an $S$-linear right admissible subcategory. Suppose that there exists an $f$-relatively ample line bundle $\mc{L}\in \operatorname{Pic}(\mc{W})$ such that for every $n\in \mathbb{Z}$ one has $\mc{L}^{\otimes n} \in \D$. Then $\D=\D^b(\mc{W})$.
    \end{lemma}
    \begin{proof}
        Since $\D\subseteq \D^b(\mc{W})$ is a right admissible subcategory, we have a semiorthogonal decomposition $\D^b(\mc{W})=\langle \D^{\perp} ,\D\rangle$.
        Let $s\in S$ be a closed point, $\mc{W}_s$ be the fiber over $s$, and $i\colon \{s\}\to S$ and $\tilde{\imath}\colon \mc{W}_s\to \mc{W}$ be the closed embeddings. For $M\in \D^{\perp}$ applying the adjunction $\tilde{\imath}^*\dashv \tilde{\imath}_*$, projection formula and isomorphisms 
        \[
        (\tilde{\imath}_*\tilde{\imath}^*(\mc{O}_{\mc{W}}))^*\cong (f^*i_*(\mc{O}_{s}))^* \cong f^*(i_*(\mc{O}_{s})^*)
        \]
        we obtain
        \begin{multline*}
        \Hom_{\D^b(\mc{W}_s)} (\tilde{\imath}^*\mc{L}^{\otimes n}[m],\tilde{\imath}^*M) \cong \Hom_{\D^b(\mc{W})} (\mc{L}^{\otimes n}[m],\tilde{\imath}_*\tilde{\imath}^*M)
        \cong \Hom_{\D^b(\mc{W})} (\mc{L}^{\otimes n}[m],\tilde{\imath}_*\tilde{\imath}^*(\mc{O}_{\mc{W}})\otimes M) \cong
        \\
        \cong
        \Hom_{\D^b(\mc{W})} (f^*(i_*(\mc{O}_{s})^*)\otimes \mc{L}^{\otimes n}[m], M) =0
        \end{multline*}
        with the last equality holding since $f^*(i_*(\mc{O}_{s})^*)\otimes \mc{L}^{\otimes n}[m]\in \mc{D}$ by the assumptions. The line bundle $\tilde{\imath}^*\mc{L} \in \operatorname{Pic}(\mc{W}_s)$ is ample, hence \cite[Corollary~3.19]{Huy06} yields $\tilde{\imath}^*M=0$, in particular, for every closed point $x\in \mc{W}$ such that $f(x)=s$ the fiber $M(x)$ is trivial. Since this holds for all closed points $s\in S$, one has $M=0$, hence $\D^{\perp}=\{0\}$ and $\D^b(\mc{W})=\D$.
    \end{proof}

\subsection{SODs for derived categories of twisted generalized flag varieties}
\label{sec:SODs_twisted_flags}

In this section we show that a full equivariantly separable-exceptional collection in the universal case of the $G$-equivariant morphism $\rfaktor{G}{P} \to \Spec k$ gives rise to a full relative separable-exceptional collection on any twisted generalized flag variety $\rfaktor{\mc{E}}{P}\to S$, and to a semiorthogonal decomposition of the derived category thereof. In particular, combined with Theorem~\ref{thm:SOD_general}, this provides semiorthogonal decompositions for the derived categories of twisted generalized flag varieties associated to semisimple groups.

In this section $G$ is an affine algebraic group over $k$, $P\leqslant G$ is a parabolic subgroup, $\pi\colon \rfaktor{G}{P}\to \Spec k$ is the structure morphism, $p\colon \mc{E}\to S$ is a right $G$-torsor over $S\in \Smk$ and $\pi_{\mc{E}}\colon\rfaktor{\mc{E}}{P}\to S$ is the corresponding relative generalized flag variety.

\begin{definition}
\label{def:descent}
We may view $\mathcal{E}$ as a variety with a left $G$-action composing the given right $G$-action with the inversion $G\xrightarrow{(-)^{-1}} G$. Faithfully flat descent gives a natural symmetric monoidal equivalence
\[
\Theta\colon \D^b_G(\mathcal{E})\xrightarrow{\simeq}\D^b(S),
\]
see e.g. \cite[Theorem~4.46]{Vi07}. If $M\in \D^b_G(\mathcal{E})$ is a locally free sheaf corresponding to a $G$-equivariant vector bundle $\mathcal{V}\to \mathcal{E}$ then $\Theta(M)$ corresponds to the vector bundle $\lfaktor{G}{\mathcal{V}} \to \lfaktor{G}{\mathcal{E}}\cong S$. For $\mc{A}\in \SAlg_G \mc{E}$ we have $\Theta(\mc{A})\in \SAlg S$ and a similar equivalence
\[
\Theta\colon \D^b_G(\mathcal{E},\mc{A})\xrightarrow{\simeq}\D^b(S,\Theta(\mc{A})).
\]

We have an isomorphism $\mathcal{E}\times G \xrightarrow{\simeq} \mathcal{E}\times_S \mathcal{E}$, $(x,g)\mapsto (x,xg)$, so we have a cartesian square
\[
\xymatrix{
\mathcal{E}\times \rfaktor{G}{P} \ar[r]^(0.55){\mu}\ar[d]_{\id \times \pi} &\rfaktor{\mathcal{E}}{P} \ar[d]^{\pi_{\mathcal{E}}}\\
\mathcal{E} \ar[r]^p & S
}
\]
with $\mu(x,gP)=xgP$. Thus $\mu$ is a pullback of a $G$-torsor hence a $G$-torsor itself, and one can check that the right $G$-action is given by 
\[
\mathcal{E}\times \rfaktor{G}{P} \times G \to \mathcal{E}\times \rfaktor{G}{P},\quad (x,gP,h)\mapsto (xh,h^{-1}gP).
\]
Similarly to the above, composing with the inversion morphism, we can make this action a left one, and faithfully flat descent gives a natural symmetric monoidal equivalence
\[
\Theta \colon \D^b_G(\mathcal{E}\times \rfaktor{G}{P}) \xrightarrow{\simeq} \D^b(\rfaktor{\mathcal{E}}{P}).
\]
Functoriality of descent yields that the following diagrams commute.
\begin{equation} \label{eq:descent}
\begin{gathered}
\xymatrix{
 \D^b_G(\mathcal{E}\times \rfaktor{G}{P}) \ar[r]^(0.55)\Theta & \D^b(\rfaktor{\mathcal{E}}{P})\\
 \D^b_G(\mathcal{E})\ar[u]^{(\id \times \pi)^*} \ar[r]^\Theta &  \D^b(S) \ar[u]_{\pi_{\mathcal{E}}^*}
}\qquad
\xymatrix{
 \D^b_G(\mathcal{E}\times \rfaktor{G}{P}) \ar[d]_{(\id \times \pi)_*} \ar[r]^(0.55)\Theta & \D^b(\rfaktor{\mathcal{E}}{P}) \ar[d]^{(\pi_{\mathcal{E}})_*}\\
 \D^b_G(\mathcal{E}) \ar[r]^\Theta &  \D^b(S) 
}
\end{gathered}
\end{equation}

    Let $q\colon \mathcal{E}\xrightarrow{p} S\to \Spec k$ be the structure morphism. Then we have the following compositions of symmetric monoidal functors
    \[
        \mathcal{L}^{\mathcal{E}}_{S} \colon \D^b_G(\Spec k)  \xrightarrow{q^*} \D^b_G(\mathcal{E}) \xrightarrow{\Theta} \D^b(S), \quad
        \mathcal{L}^{\mathcal{E}}_{\rfaktor{\mathcal{E}}{P}} \colon \D^b_G(\rfaktor{G}{P})  \xrightarrow{(q\times \id)^*} \D^b_G(\mathcal{E}\times \rfaktor{G}{P}) \xrightarrow{\Theta}  \D^b(\rfaktor{\mathcal{E}}{P}),
    \]
    and the following diagrams.
    \[
    \xymatrix{
    \D^b_G(\rfaktor{G}{P}) \ar[r]^(0.5){\mathcal{L}^{\mathcal{E}}_{S}} & \D^b(\rfaktor{\mathcal{E}}{P})  \\
    \D^b_G(\Spec k) \ar[u]^{\pi^*} \ar[r]^(0.5){\mathcal{L}^{\mathcal{E}}_{\rfaktor{\mathcal{E}}{P}}} & \D^b(S) \ar[u]_{\pi_{\mathcal{E}}^*}
    }
    \qquad
    \xymatrix{
    \D^b_G(\rfaktor{G}{P}) \ar[d]_{\pi_*} \ar[r]^(0.55){\mathcal{L}^{\mathcal{E}}_{S}} & \D^b(\rfaktor{\mathcal{E}}{P})  \ar[d]^{(\pi_{\mathcal{E}})_*} \\
    \D^b_G(\Spec k)  \ar[r]^(0.55){\mathcal{L}^{\mathcal{E}}_{\rfaktor{\mathcal{E}}{P}}} & \D^b(S) 
    }
    \]
    These diagrams commute because of the commutativity of the diagrams~\refbr{eq:descent} (for the second diagram one uses also base change property). We will usually drop the subscripts from $\mathcal{L}^{\mathcal{E}}_{S}$ and $\mathcal{L}^{\mathcal{E}}_{\rfaktor{\mathcal{E}}{P}}$ and denote both functors as $\mathcal{L}^{\mathcal{E}}$. For $A\in \SAlg_G \Spec k$ we denote
    \[
    \mc{A}^{\mc{E}}:=\mc{L}^{\mc{E}}(A) \in \SAlg S,
    \]
    and if $A$ is an equivariant Azumaya (resp. commutative separable) algebra, then $\mc{A}^\mc{E}$ is a sheaf of Azumaya (resp. commutative separable) algebras over $S$.
\end{definition}

\begin{remark}
    The passage from $\rfaktor{G}{P}\to \Spec k$ to $\lfaktor{G}{\left(\mathcal{E}\times \rfaktor{G}{P}\right)}\cong\rfaktor{\mathcal{E}}{P}\to S$ is an instance of the twisting by a torsor construction as in \cite[I.5.3]{Serre97}.
\end{remark}


\begin{lemma} \label{lem:split_homogeneous}
    Let $X\in \Smk$ be a connected proper variety homogeneous under an action of $G$. Then for some finite field extension $K/k$ there exists an isomorphism $X_K\cong \rfaktor{\bar{G}}{\bar{P}}$ for a split semisimple group $\bar{G}$ over $K$ and a parabolic subgroup $\bar{P}\leqslant \bar{G}$.
\end{lemma}
\begin{proof}
    Without loss of generality we may assume $G$ to be connected. Passing to a finite field extension we may assume that $X$ has a rational point, hence $X\cong \rfaktor{G}{P}$ for a parabolic subgroup $P\leqslant G$. Further enlarging the field we may assume that the radical $R(G_{k^{\mathrm{alg}}})$ is defined over $k$, i.e. $R(G_{k^{\mathrm{alg}}})= R(G)_{k^{\mathrm{alg}}}$. We have $R(G)\leqslant P$ by \cite[Proposition~17.19, Remark~17.31]{Mi17}, thus passing to $\rfaktor{G}{R(G)}$ we may assume $G$ to be semisimple. Finally, passing to the splitting field of a maximal torus in $G$ we obtain the claim.
\end{proof}

\begin{lemma} \label{lem:canonical_antiample}
    The relative anti-canonical bundle $\omega^*_{\pi_{\mc{E}}}$ is $\pi_{\mc{E}}$-ample.
\end{lemma}
\begin{proof}
    By \cite[Corollaire~9.6.4]{EGA43} it suffices to check ampleness of the anti-canonical bundles $\omega_{\rfaktor{\mc{E}_s}{P_{k(s)}}}^*$ of the fibers $\rfaktor{\mc{E}_s}{P_{k(s)}}$ over the closed points $s\in S$. Furthermore, by \cite[Corollaire~6.6.3]{EGA2} it suffices to check ampleness of $\omega_{(\rfaktor{\mc{E}_s}{P_{k(s)}})_K}^*$ for a finite field extension $K/k(s)$. Passing to a connected component of $\rfaktor{\mc{E}_s}{P_{k(s)}}$ and applying Lemma~\ref{lem:split_homogeneous} we obtain an isomorphism $(\rfaktor{\mc{E}_s}{P_{k(s)}})_K\cong \rfaktor{\bar{G}}{\bar{P}}$ for a split semisimple group $\bar{G}$ and a parabolic subgroup $\bar{P}\leqslant \bar{G}$ over $K$. Then the claim follows from the formula for $\omega_{\rfaktor{\bar{G}}{\bar{P}}}$ given in \cite[4.2.(6)]{Jan03} and the respective criterion for ampleness \cite[Proposition~II.4.4 and Remark~1) below]{Jan03}.
\end{proof}

\begin{lemma} \label{lem:canonical_equivariant}
    One has $\omega_{\pi_{\mc{E}}}\cong \mc{L}^{\mc{E}}(\omega_{\rfaktor{G}{P}})$ for the relative canonical bundle $\omega_{\pi_{\mc{E}}}$, in particular, $\omega_{\pi_{\mc{E}}}\in \mc{L}^{\mc{E}}(\D^b_G(\rfaktor{G}{P}))$.
\end{lemma}
\begin{proof}
    Since $p\colon \mc{E}\to S$ is a $G$-torsor, the relative canonical bundle $\omega_p$ is trivial. Thus for the composition $\mc{E}\times \rfaktor{G}{P} \xrightarrow{\pi_1} \mc{E} \xrightarrow{p} S$ there is a $G$-equivariant isomorphism $\omega_{p\circ\pi_1} \cong \pi_2^*\omega_{\rfaktor{G}{P}}$, where $\pi_1\colon \mc{E}\times \rfaktor{G}{P} \to \mc{E}$ and $\pi_2\colon \mc{E}\times \rfaktor{G}{P} \to \rfaktor{G}{P}$ are the projections. Under the descent equivalence we have $\Theta(\omega_{p\circ\pi_1}) \cong \omega_{\pi_{\mc{E}}}$, thus $\mc{L}^{\mc{E}}(\omega_{\rfaktor{G}{P}})\cong \omega_{\pi_{\mc{E}}}$.
\end{proof}

\begin{theorem} \label{thm:base_change}
    Let $A_i\in \SAlg_G \Spec k$ and $X_i\in \D^b_G(\rfaktor{G}{P},\mc{A}_i)$, $1\le i \le n$, be a full equivariantly separable-exceptional (resp. Azumaya-exceptional, resp. \'etale-exceptional, resp. exceptional) collection. Then $\mc{L}^{\mc{E}}(X_i)\in \D^b(\rfaktor{\mc{E}}{P},\mc{A}_i^\mc{E})$, $1\le i\le n$, is a full $S$-relative separable-exceptional (resp. Azumaya-exceptional, resp. \'etale-exceptional, resp. exceptional) collection.
\end{theorem}
\begin{proof}
    Consider the following commutative diagram.  
    \[
    \xymatrix{
    \D^b_G(\rfaktor{G}{P}) \ar[r]^(0.5){\mc{L}^\mc{E}} \ar[d]_{\pi_*}  & \D^b(\rfaktor{\mathcal{E}}{P}) \ar[d]^{(\pi_{\mc{E}})_*}\\
    \D^b_G(\Spec k)  \ar[r]^(0.55){\mc{L}^\mc{E}} & \D^b(S)
    }
    \]
    Since all the objects of $\D^b_G(\rfaktor{G}{P})$ are strongly dualizable and $\mc{L}^\mc{E}$ is symmetric monoidal then
    \[
    (\pi_{\mc{E}})_*\iHom_{\D^b(\rfaktor{\mathcal{E}}{P})}(\mc{L}^{\mc{E}}(X_i),\mc{L}^{\mc{E}}(X_j)) \cong (\pi_{\mc{E}})_*\mc{L}^{\mc{E}}(\iHom_{\D^b_G(\rfaktor{G}{P})}(X_i,X_j))
    \cong\mc{L}^{\mc{E}}(\pi_*\iHom_{\D^b_G(\rfaktor{G}{P})}(X_i,X_j)).
    \]
    It follows that $\mc{L}^{\mc{E}}(X_i)\in \D^b(\rfaktor{\mc{E}}{P},\mc{A}_i^\mc{E})$ is an $S$-relative $\mc{A}_i^\mc{E}$-exceptional object and $\{\mc{L}^{\mc{E}}(X_i)\}_{1\le i\le n}$ is an $S$-relative separable-exceptional (resp. Azumaya-exceptional, resp. \'etale-exceptional, resp. exceptional) collection.
    
    By fullness of the collection $\{X_i\}_{1\le i\le n}$ we have
    \[
    \D^b_G(\rfaktor{G}{P}) = \hull\left(\D^b_G(\Spec k, A_1^\op)\otimes_{A_1} X_1,\hdots,\D^b_G(\Spec k, A_n^\op)\otimes_{A_n} X_n\right).
    \] 
    It follows from the above and Proposition~\ref{prop:SOD_from_collection} that the subcategory
    \[
    \mc{D}:=\hull\left( \D^b(S,\left(\mc{A}_1^{\mc{E}}\right)^\op)\otimes_{\mc{A}_1^{\mc{E}}} \mc{L}^{\mc{E}}(X_1),\hdots,\D^b(S,\left(\mc{A}_n^{\mc{E}}\right)^\op)\otimes_{\mc{A}_n^{\mc{E}}} \mc{L}^{\mc{E}}(X_n)\right) \subseteq \D^b(\rfaktor{\mc{E}}{P})
    \]
    is admissible and $S$-linear, and we clearly have $\mc{L}^{\mc{E}}(\D^b_G(\rfaktor{G}{P}))\subseteq \D$. Lemma~\ref{lem:canonical_equivariant} yields that $\omega^{\otimes n}_{\pi_{\mc{E}}}\in \mc{D}$ for every $n\in \mathbb{Z}$, hence it follows from Lemmas~\ref{lem:generetor} and~\ref{lem:canonical_antiample} that $\D=\D^b(\rfaktor{\mc{E}}{P})$, so the collection $\mc{L}^{\mc{E}}(X_1),\mc{L}^{\mc{E}}(X_2),\hdots,\mc{L}^{\mc{E}}(X_n)$ is full.
\end{proof}

\begin{theorem} \label{thm:twisted_flags}
    Let $G$ be a quasi-split semisimple group over a field $k$, $P\leqslant G$ be a parabolic subgroup and $\mathcal{E}\to S$ be a right $G$-torsor over $S\in\Smk$. Then $\D^b(\rfaktor{\mc{E}}{P})$ admits a full $S$-relative separable-exceptional collection and, consequently, there exists an $S$-linear semiorthogonal decomposition 
    \[
    \D^b(\rfaktor{\mc{E}}{P})= \langle \D_1,\D_2,\hdots,\D_n \rangle
    \]
    with $\D_i\cong \D^b(S,\mc{A}_i)$ for some $\mc{A}_i\in \SAlg S$. Furthermore, the following holds:
    \begin{enumerate}
        \item If $G$ is split and simply connected then there exists an $S$-linear semiorthogonal decomposition as above with $\D_i\cong \D^b(S)$ and $n=[W(G):W(P)]$.
        \item If $G$ is simply connected then there exists an $S$-linear semiorthogonal decomposition as above  with $\D_i\cong \D^b(S_{k_i})$ for a finite separable extension $k_i/k$ and $\sum_{i=1}^n [k_i:k]=[W(G):W(P)]$.
        \item If $G$ is split then there exists an $S$-linear semiorthogonal decomposition  as above with all $\mc{A}_i$ being sheaves of Azumaya algebras and $n=[W(G):W(P)]$.
    \end{enumerate}
    Here $W(G)$ and $W(P)$ are the Weyl groups of $G$ and $P$ respectively.
\end{theorem}
\begin{proof}
    It was shown in {\cite[Theorems~14.2 and~14.3]{SvdK24}} that if $G$ is split and simply connected then $\D^b_G(\rfaktor{G}{P})$ admits a full equivariantly exceptional collection of length $[W(G):W(P)]$, while the existence of full equivariantly \'etale-exceptional and separable-exceptional collections in the quasi-split simply connected and in the general cases is established in Theorems~\ref{thm:SOD_rep_qs} and~\ref{thm:SOD_general} respectively in the second half of the current paper (see Section~\ref{sec:rep_translate} below for the translation between representation--theoretic and equivariant terminology). Then Theorem~\ref{thm:base_change} and Proposition~\ref{prop:SOD_from_collection} yield the claim. See also Remark~\ref{rem:equiv_Tits_prop} for the additional properties of the collections.
\end{proof}

\begin{remark} \label{rem:comments_on_SOD_twisted_flag}
\begin{enumerate}
    \item The sheaves of algebras $\mc{A}_i$ appearing in  Theorem~\ref{thm:twisted_flags} are constructed as $\mc{L}^{\mc{E}}(A_{v})$ with $A_v$ being the endomorphism algebra $\iEnd (\ind_{\tilde{B}}^{\tilde{G}} K_{ve_v})$ viewed as a $G$-representation, where $\tilde{G}$ is the simply connected cover of $G$, $\tilde{B}\leqslant \tilde{G}$ is a Borel subgroup, $K$ is the field of definition of the Steinberg weight $e_v$ corresponding to $v\in W(G)$ and $K_{ve_v}$ is the one-dimensional $K$-representation of $\tilde{B}$ of weight $ve_v$. See Definitions~\ref{def:Steinberg},~\ref{def:Azumaya_equiv} and~\ref{def:equiv_Tits} for more details.
    \item The $S$-linear semiorthogonal decomposition constructed in the proof of Theorem~\ref{thm:twisted_flags} depends on the separable-exceptional collection from Theorem~\ref{thm:SOD_general}, which in turn depends on the choice of a total order on the Weyl group $W(G)$. For some choices of the total order one may obtain additional orthogonality properties, see Remark~\ref{rem:additional_orth}.
    \item If $S=\Spec k$ then the algebras $\mc{A}_i$ are the same separable algebras (Tits algebras) that appeared in \cite[Theorem~12.2]{Pan94} in the computation of Quillen $K$-theory of twisted flag varieties. In particular, since $K$-theory takes semiorthogonal decompositions to direct sums, we recover \cite[Theorem~12.2]{Pan94}.
    \item 
        The same reasoning shows that in the notation and assumptions of Theorem~\ref{thm:twisted_flags} for $\mc{A}\in \SAlg S$ there exists a semiorthogonal decomposition
    \[
    \D^b(\rfaktor{\mc{E}}{P},\mc{A})= \langle \D_1,\D_2,\hdots,\D_n \rangle
    \]
    with $\D_i\cong \D^b(S,\mc{A}\otimes \mc{A}_i)$.
\end{enumerate}

\end{remark}    

\begin{corollary} \label{cor:Panin}
    Let $X\in\Smk$ be a smooth projective variety homogeneous under an action of a reductive group. Then there exists a semiorthogonal decomposition
    \[
        \D^b(X)= \langle \D_1,\D_2,\hdots,\D_n \rangle
    \]
    with $\D_i\cong \D^b(A_i)$ and $A_i$ being a separable algebra over $k$.
\end{corollary}
\begin{proof}
    The claim follows from Theorem~\ref{thm:twisted_flags} since it is well-known that $X\cong \rfaktor{\mc{E}}{P}$ for a torsor $\mc{E}\to \Spec k$ under a quasi-split semisimple group $G^{qs}$ and a parabolic subgroup $P\leqslant G^{qs}$. We could not find in the literature a precise reference for this well-known fact in the above form (but see e.g. the discussion in the beginning of \cite[Section~11]{Pan94} for a very close statement), so we provide the justification for the sake of completeness. Over separable closure we have $X_{\ksep}\cong G_{\ksep}/P$ for a parabolic subgroup $P\leqslant G_{\ksep}$. The center $Z(G)_{\ksep}=Z(G_{\ksep})$ is contained in every Borel subgroup \cite[Proposition~17.45]{Mi17}, so $Z(G)_{\ksep}\leqslant P$ and $Z(G)$ acts trivially on $X$. Hence passing from $G$ to $\rfaktor{G}{Z(G)}$ we may assume $G$ to be an adjoint semisimple group. There exists a unique quasi-split group $G^{qs}$ such that $G$ is an inner form of $G^{qs}$ \cite[Corollary 23.53]{Mi17}, thus there exists a $(G,G^{qs})$-bitorsor $\mc{E}\to \Spec k$ such that $G\cong \lfaktor{G^{qs}}{(\mc{E}\times G^{qs})}$ and $G^{qs}\cong \lfaktor{G}{(\mc{E}\times G)}$ for the conjugation actions of the groups on themselves. It follows that $X\cong \lfaktor{G^{qs}}{(\mc{E}\times Y)}$ for the projective variety $Y:=\lfaktor{G}{(\mc{E}\times X)}$ homogeneous under the action of $G^{qs}$. Thus it suffices to show that there is a $G^{qs}$-equivariant isomorphism $Y\cong \rfaktor{G^{qs}}{P}$ for a parabolic subgroup $P\leqslant G^{qs}$. Over the separable closure we have a $(G^{qs})_{\ksep}$-equivariant isomorphism $Y\cong \rfaktor{(G^{qs})_{\ksep}}{P}$ for a parabolic subgroup $P\leqslant (G^{qs})_{\ksep}$, and by \cite[Theorem~25.8]{Mi17} we may assume that $B_{\ksep}\leqslant P$ for a Borel subgroup $B\leqslant G^{qs}$. The variety $\rfaktor{(G^{qs})_{\ksep}}{P}$ is defined over $k$ if and only if $P$ is stable under the action of the Galois group $\Gal(\ksep/k)$, in which case $P$ is defined over $k$ and the claim follows.
\end{proof}

\begin{example} \label{ex:unitary}
    Here we give an example of the algebras that appear in the constructed semiorthogonal decompositions considering a particular instance of a unitary Grassmannian. For more details on unitary groups see \cite{KMRT98}, especially \S~2.B and \S~23 thereof.
    
    Let $L/k$ be a quadratic separable field extension, denote $\tau\in\Gal(L/k)$ the nontrivial element, and let $A$ be a central simple algebra over $L$ of dimension $16$ with a $k$-linear involution $\sigma \colon A\xrightarrow{\simeq} A^{\op}$ whose restriction to $L$ coincides with $\tau$. Recall that an ideal $I\leqslant A$ is \textit{isotropic} if $\sigma(I)\cdot I=0$. Consider the unitary Grassmannian $X:=X_1(A,\sigma)$ parametrizing isotropic right ideals $I\leqslant A$ with $\dim_L I=4$. As an algebraic variety $X$ may be realized as a closed subvariety in $\mr{Res}_{L/k} \mathrm{SB}(A)$, where $\mathrm{SB}(A)$ is the Severi-Brauer variety of dimension $4$ ideals in $A$ and $\mr{Res}_{L/k}$ is the Weil restriction functor. Over the separable closure one has $X_{\ksep}\cong \mathrm{Fl}(1,3;4)$ for the variety of flags $V\leqslant W\leqslant (\ksep)^{\oplus 4}$ with $\dim V=1$, $\dim W=3$.
    
    The unitary Grassmannian $X$ is a projective variety homogeneous under the action of the projective unitary group 
    \[
    G:=\mathrm{PGU}(A,\sigma):=\mathrm{Aut}_L(A,\sigma),
    \]
    which is an outer form of the algebraic group $\PGL_4$. Let $(\mr{M}_4(L),\sigma_h)$ be the central simple algebra of $4\times 4$ matrices over $L$ with the involution given by $\sigma_h(g):=h(\tau(g)^T)h$, where $\tau$ is applied entry-wise and 
    \[
    h:=\begin{pmatrix}
        0 & 0& 0 & 1 \\
        0 & 0& 1 & 0 \\
        0 & 1& 0 & 0 \\
        1 & 0& 0 & 0 
    \end{pmatrix}.
    \]
    The group 
    \[
    G^{qs}:=\mathrm{PGU}(\mr{M}_4(L),\sigma_h)
    \]
    is a quasi-split form of $\PGL_4$. The group $G$ is an inner form of $G^{qs}$ with
    \[
    G\cong \lfaktor{G^{qs}}{(\mc{E} \times G^{qs})}
    \]
    for the conjugation action of $G^{qs}$ on itself and for the right $G^{qs}$-torsor $\mc{E}\to \Spec k$ given by the scheme of isomorphisms $\mc{E}:=\underline{\operatorname{Iso}}_L((\mr{M}_4(L),\sigma_h),(A,\sigma))$ of $L$-algebras with involution. Similarly,
    \[
    X\cong \lfaktor{G^{qs}}{(\mc{E} \times Y)},
    \]
    with $Y$ being the variety of right ideals $I\leqslant \mathrm{M}_4(L)$ of dimension $4$ and isotropic under the involution $\sigma_h$. The variety $Y$ has a rational point given e.g. by the right ideal 
    \[
    I:=\left\{\left.
        (1,0,0,a)^T\cdot (x_1,x_2,x_3,x_4) \,\right|\, x_1,x_2,x_3,x_4\in L \right\} \leqslant \mr{M}_4(L)
    \]
    for some $a\in L$ such that $a+\tau(a)=0$. Thus $Y\cong \rfaktor{G^{qs}}{P}$ for the parabolic subgroup $P\leqslant G^{qs}$ such that $P_L$ becomes the submaximal parabolic subgroup $P_1\cap P_3\leqslant \mathrm{PGL}_4$.

    We have $\Delta(G)=A_3$ and the root system can be realised as $\{e_i- e_j\}_{i\neq j}\subseteq \Z^4$ with the simple roots $\alpha_1:=e_1-e_2$, $\alpha_2:=e_2-e_3$, $\alpha_3:=e_3-e_4$ and fundamental weights $\varpi_1:=\tfrac{3}{4}\alpha_1+\tfrac{1}{2}\alpha_2+\tfrac{1}{4}\alpha_3$, $\varpi_2:=\tfrac{1}{2}\alpha_1+\alpha_2+\tfrac{1}{2}\alpha_3$, $\varpi_3:=\tfrac{1}{4}\alpha_1+\tfrac{1}{2}\alpha_2+\tfrac{3}{4}\alpha_3$. The Weyl group $W\cong S_4$ acts permuting the coordinates and the simple reflections are given by the transpositions $s_1=(12)$, $s_2=(23)$, $s_3=(34)$. The Weyl group of the parabolic subgroup is $W_P=\{e,s_2\}$. For the nontrivial element $\tau\in\Gal(L/k)$ we have $\tau(\alpha_1)=\alpha_3$, $\tau(\alpha_2)=\alpha_2$ and $\tau(\alpha_3)=\alpha_1$, and $\tau(s_1)=s_3$, $\tau(s_2)=s_2$ and $\tau(s_3)=s_1$. It is straightforward to show that the elements of $W^P$, their fields of definitions, the corresponding shifted Steinberg weights, and the Brauer classes of algebras $A_v$ and their twists $\mc{A}^\mc{E}_v$ are as in the following table (see Definitions~\ref{def:Steinberg}, \ref{def:kv}, \ref{def:equiv_Tits} and~\ref{def:descent} for the notation).
    \[
    \begin{array}{c|c|c|c|c|c|c|c|c}
        v & s_3s_1s_2s_3s_1 & s_2s_3s_2s_1,\, s_2s_1s_2s_3 & s_2s_1s_3 & s_3s_2s_1,\, s_1s_2s_3 &  s_1s_3 & s_2s_1 ,\, s_2s_3 & s_1,\,s_3 & e\\ \hline
        ve_v & \varpi_1+\varpi_3 & \varpi_2+\varpi_3,\, \varpi_1+\varpi_2 &  \varpi_2 & \varpi_3,\, \varpi_1 &  \varpi_1+\varpi_3 & \varpi_2,\,\varpi_2 & \varpi_1,\,\varpi_3 & 0 \\
        k(v) & k & L,\,L &  k & L,\,L&  k & L,\, L & L,\,L & k \\
        {[A_v]} & k & \mr{M}_4,\,\mr{M}_4^\op & \mr{M}^\tau_6 & \mr{M}_4^\op,\,\mr{M}_4 &  k & \mr{M}_6,\,\mr{M}_6 & \mr{M}_4,\,\mr{M}_4^\op & k \\
        {[\mc{A}_v^{\mc{E}}]} & k & A,\,A^\op &  B & A^\op,\,A &  k & B_L,\, B_ L & A,\,A^\op & k
    \end{array}
    \]
    Here we group together the elements of $W^P$ that are permuted by $\tau$ and use the following notation:
    \begin{itemize}
        \item 
        $\mr{M}_4:=\mr{M}_4(L)$ equipped with the tautological action of $G^{qs}=\mr{Aut}_L(\mr{M}_4(L),\sigma_h)$,
        \item
        $\mr{M}^\tau_6 = \End_k W$ for the unique irreducible dimension $6$ representation $W$ over $k$ of the simply connected cover $\mathrm{SU}(\mathrm{M}_4(L),\sigma_h)$ of $\mathrm{PGU}(\mathrm{M}_4(L),\sigma_h)$,
        \item 
        $\mr{M}_6:=\mr{M}^\tau_6 \otimes_k L$, $B:= \lfaktor{G^{qs}}{(\mc{E}\times \mr{M}^\tau_6})$, $B_L:=B\otimes_k L$. Note that $B_L$ is Brauer equivalent to $A\otimes_L A$.
    \end{itemize}
    By the proof of Theorem~\ref{thm:twisted_flags} we have a semiorthogonal decomposition
    \[
    \D^b(X) = \langle \Phi_1(\D^b(k)), \Phi_2(\D^b(A)), \Phi_3(\D^b(B)), \Phi_4(\D^b(A)), \Phi_5(\D^b(k)), \Phi_6(\D^b(B_L)), \Phi_7(\D^b(A)), \Phi_8(\D^b(k))\rangle
    \]
    for certain fully faithful functors $\Phi_i$.
\end{example}

\section{Decompositions for representation categories}

\subsection{Relative exceptional collections in the representation--theoretic setting} 
\label{sec:rep_translate}
In this section we recall some generalities on representations of affine algebraic groups, translate the notions of Sections~\ref{sec:recollection_derived} and~\ref{sec:rel_ex} into this setting and set up some notation.

In this section $G$ is an affine algebraic group over $k$ and $P\leqslant G$ is a parabolic subgroup.

\begin{definition} \label{def:rep_rec}
Let $K/k$ be a finite separable field extension. We use the following notation.
\begin{itemize}
    \item $\lrep{K} G$ is the symmetric monoidal abelian category of finite dimensional $K$-representations of $G$. If $K=k$ we put $\rep G:= \lrep{k} G$, and under this notation we have $\lrep{K} G = \rep G_K$.
    \item $\Alg_G K$ the category of monoids in $\lrep{K} G$, that is, the category of finite dimensional algebras over $K$ equipped with $G$-action. $\SAlg_G K$ is the full subcategory of $\Alg_G K$ consisting of the algebras that are separable after forgetting the action of $G$, and $\CSAlg_G K$ is the full category consisting of the algebras that are central simple algebras over $K$ after forgetting the $G$-action.
    \item For $A\in \SAlg_G K$ we denote $\lrep{A} G$ the category of left $A$-modules in $\lrep{K} G$, that is the category of finite dimensional $K$-representations of $G$ equipped with an equivariant left $A$-action. 
    \item $\D^b(\lrep{K} G)$ and $\D^b(\lrep{A} G)$ are the respective bounded derived categories.
\end{itemize} 

For a homomorphism of affine algebraic groups $H\to G$ over $k$ restriction of representations is an exact symmetric monoidal functor $\res^G_H\colon \lrep{K} G \to \lrep{K} H$. For $A\in\SAlg_G K$ we denote $\lrep{A} H:=\lrep{(\res^G_H A)} H$ and 
\[
\res^G_H\colon \D^b(\lrep{A} G) \to \D^b(\lrep{A} H)
\]
the corresponding functor between the derived categories. Below we often omit the restriction functor from the notation. For a parabolic subgroup $P\leqslant G$ the functor
$\res^G_P\colon \D^b(\lrep{A} G) \to \D^b(\lrep{A} P)$ admits a right adjoint denoted
\[
\ind^G_P\colon \D^b(\lrep{A} P) \to \D^b(\lrep{A} G).
\]
We usually omit the restriction functor $\res^G_P$ from the notation, so for $M\in \D^b(\rep G)$ and $X\in \D^b(\rep P)$ we have $M\otimes X:= \res^{G}_P(M)\otimes X$.

Recall that for a subgroup $H\leqslant G$ there is a canonical equivalence $\rep H \cong \Coh_G \rfaktor{G}{H}$ that for $A\in\SAlg_G k$ induces further equivalences $\lrep{A} H\cong \Coh_G (\rfaktor{G}{H},A)$ and $\D^b(\lrep{A} H)\cong \D^b_G(\rfaktor{G}{H},A)$. Under these equivalences restriction and induction of representations correspond to the pullback and direct image functor of sheaves respectively, i.e. the following diagrams commute.
\[
\xymatrix{
 \D^b(\lrep{A} H) \ar[r]^\simeq &  \D^b_G(\rfaktor{G}{H},A)\\
 \D^b(\lrep{A} G)\ar[u]^{\res^G_H} \ar[r]^\simeq &  \D^b_G(\Spec k,A) \ar[u]_{\pi^*}
}
\qquad
\xymatrix{
 \D^b(\lrep{A} P) \ar[r]^\simeq \ar[d]_{\ind^G_P} &  \D^b_G(\rfaktor{G}{P},A) \ar[d]^{\pi_*}\\
 \D^b(\lrep{A} G) \ar[r]^\simeq &  \D^b_G(\Spec k,A) 
}
\]
Here $\pi$ is the structure morphism. In particular, for the categories $\D^b(\lrep{A} G)$ we have the usual formalism of $\otimes$, $\iHom$, $\res^G_P$, $\ind^G_P$, projection formula and base change as in Proposition~\ref{prop:functor_formalism}. 
\end{definition}

\begin{remark}
    For a finite separable field extension $K/k$ we have $K\in \SAlg_G k$ viewed as a separable algebra with trivial $G$-action, and the category $\lrep{K} G$ may be viewed either as the category of finite dimensional $K$-representations or finite dimensional $k$-representations equipped with an equivariant left $K$-action with $G$ acting on $K$ trivially. Below when we write $\lrep{K} G$ for a finite separable field extension $K/k$ we always assume that $G$ acts on $K$ trivially.
\end{remark}

\begin{definition} \label{def:gamma_rep}
    Let $K/k$ be a finite separable field extension. Recall that we have canonical equivalences 
    \[
    \D^b(\lrep{K} G) \cong \D^b_G(\Spec k, K)\cong \D^b_G(\Spec K).
    \]
    We denote
    \[
    (r_{K})_*:=(r_{K/k})_*: \D^b(\lrep{K} G) \to \D^b(\rep G)
    \]
    the restriction of scalars functor, which under the above equivalences coincides with the forgetful functor $\D^b_G(\Spec k, K) \to \D^b_G(\Spec k)$ and with the direct image functor $\D^b_G(\Spec K) \to \D^b_G(\Spec k)$ for the projection $r_K\colon \Spec K\to \Spec k$.

    Let $L/k$ be a finite Galois field extension and denote $\Gamma:=\Gal(L/k)$. For $\gamma\in\Gamma$ we have an autoequivalence $(\Spec \gamma)_*\colon \D^b(\lrep{L} G)\xrightarrow{\simeq} \D^b(\lrep{L} G)$ by taking direct image along the automorphism of $\Spec L$ induced by $\gamma$. Note that $(\Spec \gamma)_*=(\Spec \gamma^{-1})^*$. For $M\in \D^b(\lrep{L} G)$ we put 
    \[
    M^\gamma:=(\Spec \gamma)_*(M),
    \]
    this endows $M\in \D^b(\lrep{L} G)$ with a right action of $\Gamma$ (note that $\Gamma$ tautologically acts on the left on $L$, thus on the right on $\Spec L$, and direct image is covariant).
\end{definition}

Below we repeatedly use the following immediate properties of extension and restriction of scalars.
\begin{lemma} \label{lem:pushpull}
    Let $k\subseteq K\subseteq L$ be a tower of finite separable field extensions with $L/k$ being Galois and denote $\Gamma:=\Gal(L/k)$, $\Gamma_K:=\Gal(L/K)\leqslant \Gamma$. Then the following holds.
        \begin{enumerate}
        \item For $M\in \D^b(\lrep{K} G)$ there is an isomorphism
        \[
        (r_{K})_*(M)\otimes_k L \cong \bigoplus\limits_{\gamma\in \lfaktor{\Gamma_K}{\Gamma}} (M\otimes_K L)^\gamma,
        \]
        with the direct sum taken for a set of representatives of left cosets $\lfaktor{\Gamma_K}{\Gamma}$,
        \item
        For $M\in \D^b(\rep G)$ there is an isomorphism 
        \[
        (r_K)_*(M\otimes K) \cong M^{\oplus [K:k]}.
        \]
    \end{enumerate}
    In particular, extension of scalars is conservative.
\end{lemma}
\begin{proof}
    The first claim follows from base change for the cartesian square
    \[
        \xymatrix{
        \bigsqcup\limits_{\gamma\in \lfaktor{\Gamma_K}{\Gamma}} \Spec L \ar[rr]^(0.55){\bigsqcup\limits_{\gamma\in \lfaktor{\Gamma_K}{\Gamma}} \Spec \gamma} \ar[d] & & \Spec L \ar[d] \\
        \Spec K \ar[rr]^{r_K} & & \Spec k
        }
    \]

    The second claim follows from the projection formula $(r_K)_*(M\otimes K) \cong M\otimes K$, where $K$ on the right-hand side is considered as a trivial $k$-representation of $G$ of dimension $[K:k]$.
\end{proof}

The following definition is the translation of Definition~\ref{def:relative_decomposition} into representation--theoretic setting.
\begin{definition} \label{def:rel_rep}
    Let $A\in \SAlg_G k$. An object $X\in \D^b(\lrep{A} P)$ is \textit{$G$-relative $A$-exceptional} if the canonical homomorphism $A\to \ind^G_P\iEnd_{\D^b(\rep P)} X$ is an isomorphism. A $G$-relative $A$-exceptional $X$ is also called \textit{$G$-relative separable-exceptional}. If $A\in \CSAlg_G k$ (resp. $A$ is commutative, resp. $A=k$ with the trivial $G$-action) then a $G$-relative $A$-exceptional $X$ is \textit{$G$-relative Azumaya-exceptional (resp. \'etale-exceptional, resp. exceptional)}.
    
    A sequence of objects $X_i\in \D^b(\lrep{A_i} P)$, $1\le i\le n$, is called a \textit{$G$-relative separable-exceptional (resp. Azumaya-exceptional, resp. \'etale-exceptional, resp. exceptional) collection} if 
    \begin{enumerate}
        \item $X_i$ is $G$-relative separable-exceptional (resp. Azumaya-exceptional, resp. \'etale-exceptional, resp. exceptional) for every $1\le i\le n$,
        \item $\ind_P^G \iHom_{\D^b(\rep P)}(X_i,X_j) =0$ for every $1\le j<i\le n$.
    \end{enumerate}

    A $G$-relative separable-exceptional collection $\{X_i\}_{1\le i\le n}$, is \textit{full} if 
    \[
    \D^b(\rep P) = \hull \left( \D^b(\lrep{A_1^\op} G) \otimes_{A_1} X_1,\hdots, \D^b(\lrep{A_n^\op} G) \otimes_{A_n} X_n\right),
    \]
    where on the right stands the smallest triangulated subcategory of $\D^b(\rep P)$ containing $M_i\otimes_{A_i} X_i$ for all $M_i\in \D^b(\lrep{A_i^\op} G)$, $1\le i \le n$.
    \end{definition}

    \begin{proposition} \label{prop:SOD_from_collection_rep}
        Let $X_i\in \D^b(\lrep{A_i} P)$, $1\le i\le n$, be a full $G$-relative separable-exceptional collection. Then
        \begin{enumerate}
            \item the functors $\Phi_i\colon \D^b(\lrep{A_i^{\op}} G) \to \D^b(\rep P)$, $M\mapsto M\otimes_{A_i} X_i$, are fully faithful,
            \item there is a semiorthogonal decomposition
            \[
            \D^b(\rep P)=\langle \Phi_1(\D^b(\lrep{A_1^{\op}} G)),\hdots, \Phi_n(\D^b(\lrep{A_n^{\op}} G))\rangle.
            \]
        \end{enumerate}
    \end{proposition}
    \begin{proof}
        These are items 1. and 4. of Proposition~\ref{prop:SOD_from_collection} in the notation of Definitions~\ref{def:rep_rec} and~\ref{def:rel_rep}.
    \end{proof}

\subsection{Relative decompositions for simply connected split groups}   \label{sec:split_sc}
In this section we recall the main results and constructions of \cite{SvdK24}, show some additional orthogonality statements for the semiorthogonal decompositions of loc. cit. and prove a result on partial independence of the constructed exceptional sequence from the choice of a total order on the Weyl group.

In this section $G$ is a simply connected split semisimple group over $k$, and $T\leqslant B\leqslant P\leqslant G$ is a split maximal torus, a Borel subgroup, and a parabolic subgroup respectively. The choice of a Borel subgroup determines the set of simple roots $\Pi$, the chamber of dominant weights $X^*(T)_+\subseteq X^*(T):=\Hom(T,\mathbb{G}_m)$ and fundamental weights $\varpi_{\alpha}\in X^*(T)_+$, $\alpha\in \Pi$, and following \cite[II.1.8]{Jan03} and \cite{SvdK24}  we do this in such a way that $B$ corresponds to the negative roots, i.e. $B$ is generated by $T$ and unipotent subgroups $U_\alpha\leqslant G $ for negative roots $\alpha$. We denote $W:=W(G)$ and $W_P:=W(P)$ the Weyl groups of $G$ and $P$ respectively, $\lebru$ the Bruhat order on $W$ and $W^P$ the set of minimal (in the Bruhat order) representatives for cosets $\rfaktor{W}{W_P}$ \cite[\S~1.10]{Hu92}. The length function on $W$ with respect to the simple reflections $s_\alpha$, $\alpha\in\Pi$, is denoted $l(-)$.

\begin{definition} \label{def:Steinberg}
    For $v\in W$ denote
    \[
    e_v := v^{-1}\sum_{\substack{\alpha\in \Pi\\ v^{-1}\alpha<0}} \varpi_\alpha 
    \]
    the \textit{Steinberg weight} associated to $v$ \cite[\S~2.1]{St75}. Here the sum is taken over all simple roots $\alpha$ such that $v^{-1}\alpha$ is a negative root.
\end{definition}

\begin{definition} \label{def:Pv}
    For $w\in W$ let 
    \[
    Z_w:=\rfaktor{\overline{BwB}}{B}\subseteq \rfaktor{G}{B}
    \]
    be the corresponding \textit{Schubert variety} and put
    \[
    \partial Z_w:= \bigcup_{Z_v\subsetneq Z_w} Z_v.
    \]
    Both $Z_w$ and $\partial Z_w$ are $B$-invariant.

    Let $\lambda\in X^*(T)_+$ be a dominant character, $
    \mathscr{L}\colon \rep B \xrightarrow{\simeq} \Coh_{G} \rfaktor{G}{B} \xrightarrow{} \Coh_{B} \rfaktor{G}{B}$
    be the composition of the canonical equivalence with the restriction of the action from $G$ to $B$, and let $\mathscr{L}(k_{\lambda})$ be the $B$-equivariant line bundle over $\rfaktor{G}{B}$ associated to the one-dimensional representation $k_{\lambda}$ of $B$ with the character $\lambda$. Restricting $\mathscr{L}(k_{\lambda})$ to a closed $B$-invariant subvariety and taking global sections we obtain a representation of $B$. In particular, for $\mu\in X^*(T)$ we have the following representations (see \cite[\S~1.3]{vdK89}, \cite[\S~2.3]{vdK93} and \cite[{\S~3.2--3.3}]{SvdK24} for more details).
    \begin{itemize}
        \item \textit{Dual Joseph module} 
        \[
        \mc{P}(\mu):=\mathrm{H}^0\left( Z_{w^{-1}},\mathscr{L}(k_{w\mu})|_{Z_{w^{-1}}}\right),
        \]
        where $w\in W$ is such that $w\mu \in X^*(T)_+$. The representation up to an isomorphism does not depend on the choice of $w$ \cite[Lemma~2.3.1]{vdK93}.
        \item
        \textit{Relative Schubert module}
        \[
        \mc{Q}(\mu):=\ker \left(\mathrm{H}^0( Z_{w^{-1}},\mathscr{L}(k_{w\mu})|_{Z_{w^{-1}}}) \to \mathrm{H}^0(\partial Z_{w^{-1}},\mathscr{L}(k_{w\mu})|_{\partial Z_{w^{-1}}})\right),
        \]
        where $w\in W$ is the minimal element of $W$ such that $w\mu \in X^*(T)_+$.
    \end{itemize}
    We are mostly interested in the modules $\mc{P}(-e_v)^*, \mc{Q}(e_v) \in \rep B$, $v\in W$, with $e_v\in X^*(T)$ being the Steinberg weights. Put
    \[
    \mc{P}_v:=\mc{P}(-e_{v}),\quad \hat{\mc{P}}_v:=\ind^{P}_{B} \mc{P}_v,\quad 
    \mc{Q}_v:=\mc{Q}(e_{v}),\quad \hat{\mc{Q}}_v:=\ind^{P}_{B} \mc{Q}_v.
    \]
    Recall that by \cite[Lemma~14.15]{SvdK24} for $v\in W^P$ the counit homomorphism $\res_{B}^{P} \ind_{B}^{P} \mathcal{P}_v \to \mathcal{P}_v$ is an isomorphism, in other words, the $B$-representation $\mathcal{P}_v$ is a restriction of the $P$-representation $\hat{\mathcal{P}}_v$. Alternatively, this also follows from the observation that $v^{-1}w_0\cdot(-e_v)\in X^*(T)_+$ and $Z_{v^{-1}w_0}$ is $P$-invariant, which in turn is an immediate consequence of the Bruhat decomposition for $P$ and minimality of $v$.
\end{definition}
    
\begin{definition} \label{def:Xv}
    A subset $I\subseteq W^P$ is an \textit{initial segment} in the Bruhat order if for every $v\in I$ and $w\in W^P$ such that $w\lebru v$ one has $w\in I$. For an initial segment $I\subseteq W^P$ we put $\bar{I}:=W^P\setminus I$ and
    \[
    \hat{\mc{P}}_{I}:=\hull\left(\left\{ M\otimes \hat{\mc{P}}^*_w \right\}_{M\in \rep G,\, w\in I}\right) \subseteq \D^b(\rep P),\quad
    \hat{\mc{Q}}_{\bar{I}}:=\hull\left(\left\{M\otimes \hat{\mc{Q}}_w \right\}_{M\in \rep G,\, w\in \bar{I}}\right)\subseteq \D^b(\rep P).
    \]
    Then \cite[Theorems~14.21 and~14.22]{SvdK24} yield a semiorthogonal decomposition
    \begin{equation}\label{eq:SOD_PQ}
        \D^b(\rep P)= \left\langle \hat{\mc{Q}}_{\bar{I}}, \hat{\mc{P}}_{I}\right\rangle.
    \end{equation}
    Thus the inclusion $(i_{I})_*\colon \hat{\mc{P}}_{I}\to \D^b(\rep P)$ admits a right adjoint $i_{I}^!\colon \D^b(\rep P)\to \hat{\mc{P}}_{I}$ while the inclusion $(j_{\bar{I}})_*\colon \hat{\mc{Q}}_{\bar{I}}\to \D^b(\rep P)$ admits a left adjoint $j_{\bar{I}}^*\colon \D^b(\rep P)\to \hat{\mc{Q}}_{\bar{I}}$.

    Let $\le$ be a total order on $W$ extending the Bruhat order $\lebru$ and for $v\in W^P$ put
    \[
    \begin{gathered}    
    I_{\le v}:=\{w\in W^P\,|\, w\le v\},\quad I_{>v}:=\bar{I}_{\le v}=\{w\in W^P\,|\, v <w\}, \\
    I_{< v}:=\{w\in W^P\,|\, w< v\},\quad I_{\ge v}:=\bar{I}_{< v}=\{w\in W^P\,|\, v \le w\}.
    \end{gathered}    
    \]
    Below we omit the letter $I$ from the notation and write $(i_{\le v})_*:= (i_{I_{\le v}})_*$, $\hat{\mc{P}}_{\le v}:= \hat{\mc{P}}_{I_{\le v}}$, etc. For $v\in W^P$ following \cite[\S~10.2 and~\S~14]{SvdK24} we denote
    \[
    \hat{X}_v:=(j_{\ge v})_*j_{\ge v}^* \hat{\mc{P}}_v^*\in \D^b(\rep P),\quad \hat{Y}_v:=(i_{\le v})_*i_{\le v}^! \hat{\mc{Q}}_v\in \D^b(\rep P).
    \]
    By \cite[\S~11 and Proof of Theorems~14.2,~14.3 and~14.5 on page 60]{SvdK24} we have $\hat{X}_v\cong \hat{Y}_v$.
\end{definition}

\begin{theorem}[{\cite[Theorems~14.2 and~14.3]{SvdK24}}] \label{thm:SOD_SvdK}
    Let $\le$ be a total order on $W$ extending the Bruhat order, enumerate the elements of $W^P$ as $v_n<v_{n-1}<\hdots<v_1$, $n:=\#W^P$, and put $\hat{X}_i:=\hat{X}_{v_i}$ in the notation of Definition~\ref{def:Xv}. Then $\{\hat{X}_i\}_{1\le i\le n}$ is a full $G$-relative exceptional collection in $\D^b(\rep P)$, the functors 
    \[
    \Phi_i\colon \D^b(\rep G) \to \D^b(\rep P), \quad M\mapsto  M\otimes \hat{X}_i,
    \]
    are fully faithful and there is a semiorthogonal decomposition
        \[
        \D^b(\rep P) = \langle \Phi_1(\D^b(\rep G)), \Phi_2(\D^b(\rep G)), \hdots, \Phi_{n}(\D^b(\rep G))\rangle.
        \]
\end{theorem}

\begin{definition} \label{def:ant}
    Let $(-,-)$ be a $W$-invariant inner product on $X^*(T)\otimes_{\mathbb{Z}} \mathbb{R}$ and $\|-\|$ be the associated norm. For $\lambda,\mu \in X^*(T)$ we say that $\lambda \leant \mu$ if either $\|\lambda\|<\|\mu\|$ or $\lambda =w\nu$, $\mu =v\nu$ with $\nu\in X^*(T)_+$ and $v \lebru w$. Following \cite[\S~1.2]{vdK89} we refer to this order as \textit{antipodal excellent order}. 
    
    For $v\in W^P$ we put
    \[
    J_v:=\{w\in W^P\,|\, e_w\leant e_v\},
    \]
    cf. the definition of $\rm{pre}_{\lambda}$ in the beginning of \cite[Section 11]{SvdK24}.
\end{definition}

\begin{remark}
    It is immediate to see that the above definition is equivalent to the one from \cite[\S~1.2]{vdK89}, since $w\lambda \in X^*(T)_+$ if and only if $-w_0w\lambda\in X^*(T)_+$ for the longest element $w_0\in W$. Furthermore, it is straightforward to check that this definition is also equivalent to \cite[Definition~2]{Ana12}.
\end{remark}

The following variant of the semiorthogonal decomposition~\refbr{eq:SOD_PQ} was essentially proved in \cite{SvdK24}.
\begin{theorem}[{\cite{SvdK24}}] \label{thm:SOD_PQ_with_ant}
    Let $I\subseteq W^P$ be an initial segment in the Bruhat order and $v\in W^P$, put
    \[
    \begin{gathered}
        \hat{\mc{Q}}_{\bar{I}\cap J_v}:=\hull\left(\left\{M\otimes \hat{\mc{Q}}_w \right\}_{M\in \rep G,\, w\in \bar{I}\cap J_v}\right) \subseteq \D^b(\rep P),
        \\    
        \hat{\mc{P}}_{I\cap J_v}:= \hull\left(\left\{M\otimes \hat{\mc{P}}^*_w \right\}_{M\in \rep G,\, w\in I\cap J_v}\right) \subseteq \D^b(\rep P),
    \end{gathered}   
    \]
    and let $\D_{\leant e_v}\subseteq \D^b(\rep P)$ be the full triangulated subcategory of $\D^b(\rep P)$ generated by $M \otimes N$ with $M\in \rep G$ and $N\in \rep P$ such that all weights $\lambda$ of $N$ satisfy $\lambda \leant e_v$. Then there is a semiorthogonal decomposition
    \begin{equation} \label{eq:SOD_PQ_with_ant}
        \D_{\leant e_v}=\left\langle \hat{\mc{Q}}_{\bar{I}\cap J_v},\hat{\mc{P}}_{I\cap J_v} \right\rangle.
    \end{equation}
\end{theorem}
\begin{proof}
    It follows from \cite[Theorem~1.6]{vdK89} that for $w\in W^P$ all weights $\lambda$ of $\mc{P}(-e_w)^*$ and $\mc{Q}(e_w)$ satisfy $\lambda\leant e_w$, and by \cite[Lemma~14.16]{SvdK24} the same holds for the weights of $\hat{\mc{P}}_w^*$ and $\hat{\mc{Q}}_w$. Hence $\hat{\mc{Q}}_{\bar{I}\cap J_v},\hat{\mc{P}}_{I\cap J_v}\subseteq \D_{\leant e_v}$.
    
    The fact that $\D_{\leant e_v}=\hull\left( \hat{\mc{Q}}_{\bar{I}\cap J_v},\hat{\mc{P}}_{I\cap J_v} \right)$ follows from \cite[Theorem~14.18]{SvdK24}, or rather from the proof of the loc. cit., since the same proof works when one changes "$<_{\mathrm{a}} \tau$" in the statement to "$\leant \tau$" (with $\tau=e_v$ in our case), or, more generally, to "$\in \mathfrak{I}$" with $\mathfrak{I}\subseteq X^*(T)$ being an initial segment with respect to $\leant$.
    
    The inclusion $\hat{\mc{Q}}_{\bar{I}\cap J_v}\subseteq \left(\hat{\mc{P}}_{I\cap J_v}\right)^\perp$ follows from \cite[Theorem~14.22]{SvdK24}. 
\end{proof}

\begin{proposition} \label{prop:XY_ant}
    In the notation of Definition~\ref{def:Xv} and Theorem~\ref{thm:SOD_PQ_with_ant} we have 
    \[
    (j_{\bar{I}})_*j_{\bar{I}}^*(\hat{\mc{P}}_v^*) \in \hat{\mc{Q}}_{\bar{I}\cap J_v}, \quad (i_I)_*i_I^!(\hat{\mc{Q}}_v) \in \hat{\mc{P}}_{I\cap J_v}.
    \]
\end{proposition}
\begin{proof}
    Since $\hat{\mc{Q}}_{\bar{I}\cap J_v} \subseteq \hat{\mc{Q}}_{\bar{I}}$ and $\hat{\mc{P}}_{I\cap J_v}\subseteq \hat{\mc{P}}_{I}$, the semiorthogonal decompositions~\refbr{eq:SOD_PQ} and~\refbr{eq:SOD_PQ_with_ant} are compatible in a sense that for $N\in \D_{\leant e_v}$ there are isomorphisms
    \[
    (j_{\bar{I}})_*j_{\bar{I}}^* h (N)\cong h(j_{\bar{I}\cap J_v})_*j_{\bar{I}\cap J_v}^* (N),\quad 
    (i_{I})_*i_{I}^!h(N)\cong h(i_{I\cap J_v})_*i_{I\cap J_v}^! (N),
    \]
    where $(j_{\bar{I}\cap J_v})_*\colon \hat{\mc{Q}}_{\bar{I}\cap J_v}\to \D_{\leant e_v}$, $(i_{I\cap J_v})_*\colon \hat{\mc{P}}_{I\cap J_v}\to \D_{\leant e_v}$ and $h\colon \D_{\leant e_v} \to \D^b(\rep P)$ are the inclusions, and $j_{\bar{I}\cap J_v}^*$ and $i_{I\cap J_v}^!$ are the respective left and right adjoint functors. The same argument as in the beginning of the proof of Theorem~\ref{thm:SOD_PQ_with_ant} shows that $\hat{\mc{P}}_v^*,\hat{\mc{Q}}_v\in \D_{\leant e_v}$. The claim follows.
\end{proof}

\begin{corollary} \label{cor:orthogonal}
    In the notation of Theorem~\ref{thm:SOD_SvdK} let $1\le i<j\le n$ and suppose that 
    \[
    [v_j,v_i]\cap J_{v_i}\cap J_{v_j}=\emptyset,
    \]
    where $[v_j,v_i]:=\{v_l\,|\, l\in [i,j]\}\subseteq W^P$. Then for every $M_1,M_2\in \D^b(\rep G)$ one has
    \[
        \Hom_{\D^b(\rep P)}(M_1\otimes \hat{X}_i,M_2\otimes \hat{X}_j)=0,\qquad
        \Hom_{\D^b(\rep P)}(M_2\otimes \hat{X}_j,M_1\otimes \hat{X}_i)=0.
    \]
    In particular, if for every $w\in [v_j,v_i]$ one has 
    \begin{equation} \label{eq:lengths}
    l(v_i)=l(v_j)=l(w),\quad \|e_{v_i}\|=\|e_{v_j}\|=\|e_{w}\|,        
    \end{equation}
    then the subcategories $\Phi_l(\D^b(\rep G))\subseteq \D^b(\rep P)$, $i\le l\le j$, are pairwise orthogonal.
\end{corollary}
\begin{proof}
    The second vanishing follows from the semiorthogonality property in Theorem~\ref{thm:SOD_SvdK}. For the first one recall that 
    \[
    \hat{X}_i\cong (i_{\le v_i})_*i_{\le v_i}^!(\hat{\mc{Q}}_{v_i}),\quad
    \hat{X}_j= (j_{\ge v_j})_*j_{\ge v_j}^*(\hat{\mc{P}}_{v_j}^*),
    \]
    hence Proposition~\ref{prop:XY_ant} yields
    \[
    M_1\otimes \hat{X}_i \in \hat{\mc{P}}_{(I_{\le v_i})\cap J_{v_i}},\quad M_2\otimes \hat{X}_j\in \hat{\mc{Q}}_{(I_{\ge v_j})\cap J_{v_j}}.
    \]
    Then it suffices to show that for every $N_1,N_2\in \D^b(\rep G)$ and $w_1\in (I_{\le v_i})\cap J_{v_i}$, $w_2\in (I_{\ge v_j})\cap J_{v_j}$ we have
    \begin{equation} \label{eq:PQ_orth}
        \Hom_{\D^b(\rep P)}(N_1\otimes \hat{\mc{P}}^*_{w_1}, N_2\otimes \hat{\mc{Q}}_{w_2})=0.
    \end{equation}
    Using duality and res-ind adjunction we obtain
    \[
        \Hom_{\D^b(\rep P)}(N_1\otimes \hat{\mc{P}}^*_{w_1},N_2\otimes \hat{\mc{Q}}_{w_2})
        \cong\Hom_{\D^b(\rep G)}(N_1\otimes N_2^*, \ind^G_P (\hat{\mc{P}}_{w_1}\otimes \hat{\mc{Q}}_{w_2})).
    \]
    By \cite[Theorem~14.22.2]{SvdK24} we have $\ind^G_P (\hat{\mc{P}}_{w_1}\otimes \hat{\mc{Q}}_{w_2})=0$ unless $w_2\lebru w_1$. Thus in order for vanishing~\refbr{eq:PQ_orth} to fail we should have $w_1,w_2\in W^P$ such that
    \[
    w_1\le v_i,\quad w_2\ge v_j,\quad w_2\lebru w_1,\quad e_{w_1}\leant e_{v_i},\quad e_{w_2}\leant e_{v_j}.
    \]
    We claim that no such $w_1,w_2$ exist. Indeed, since $\le$ refines the Bruhat order,  the first three inequalities yield $w_1,w_2\in [v_j,v_i]$. If $\|e_{w_2}\|< \|e_{w_1}\|$ then $e_{w_2}\leant e_{w_1}$ and $w_2\in [v_j,v_i]\cap J_{v_i}\cap J_{v_j}$, otherwise ${e_{w_1}\leant e_{w_2}}$ and $w_1\in [v_j,v_i]\cap J_{v_i}\cap J_{v_j}$. In both cases this leads to a contradiction since by the assumption ${[v_j,v_i]\cap J_{v_i}\cap J_{v_j}=\emptyset}$. Hence the vanishing~\refbr{eq:PQ_orth} holds as claimed.

    For the last claim note that the assumption~\refbr{eq:lengths} yields that for every $w\in [v_j,v_i]\setminus \{v_i\}$ we have $e_{w}\not\leant e_{v_i}$, and also $e_{v_i}\not\leant e_{v_j}$, thus $[v_j,v_i]\cap J_{v_i}\cap J_{v_j}=\emptyset$.
\end{proof}

\begin{proposition} \label{prop:projections_independent}
    Let $I_1,I_2\subseteq W^P$ be initial segments in the Bruhat order and $v\in W^P$. Suppose that $I_1\cap J_v= I_2\cap J_v$. Then
    \[
    (j_{\bar{I}_1})_*j_{\bar{I}_1}^*(\hat{\mc{P}}_v^*) \cong (j_{\bar{I}_2})_*j_{\bar{I}_2}^*(\hat{\mc{P}}_v^*),\quad
    (i_{I_1})_*i_{I_1}^!(\hat{\mc{Q}}_v)\cong (i_{I_2})_*i_{I_2}^!(\hat{\mc{Q}}_v).
    \]
\end{proposition}
\begin{proof}
    The same argument as in the proof of Proposition~\ref{prop:XY_ant} shows that
    \[
        (j_{\bar{I_1}})_*j_{\bar{I}_1}^* h (\hat{\mc{P}}_v^*)\cong h (j_{\bar{I}_1\cap J_v})_*j_{\bar{I}_1\cap J_v}^* (\hat{\mc{P}}_v^*),\quad 
    (j_{\bar{I_2}})_*j_{\bar{I}_2}^* h (\hat{\mc{P}}_v^*)\cong h (j_{\bar{I}_2\cap J_v})_*j_{\bar{I}_2\cap J_v}^* (\hat{\mc{P}}_v^*),
     \]
    with the notation as in the proof of Proposition~\ref{prop:XY_ant}. Since the right-hand sides are isomorphic by assumption, the first claimed isomorphism follows. The second one is obtained by a similar reasoning.
\end{proof}

\begin{corollary} \label{cor:Xv_independent}
    Let $\le$ and $\le'$ be total orders on $W$ extending the Bruhat order $\lebru$. For $v\in W^P$ consider $\hat{X}_v,\hat{X}'_v\in \D^b(\rep P)$ constructed as in Definition~\ref{def:Xv} for $\le$ and $\le'$ respectively. Suppose that ${I_{\le v}\cap J_v = I_{\le' v}\cap J_v}$. Then $\hat{X}_v\cong \hat{X}'_v$.
\end{corollary}
\begin{proof}
    Recall that $\hat{X}_v\cong (i_{I_{\le v}})_*i_{I_{\le v}}^!(\hat{\mc{Q}}_v)$ and $\hat{X}'_v\cong (i_{I_{\le' v}})_*i_{I_{\le' v}}^!(\hat{\mc{Q}}_v)$. Since $I_{\le v}\cap J_v = I_{\le' v}\cap J_v$ by assumption, the claim follows from Proposition~\ref{prop:projections_independent}.
\end{proof}

\begin{proposition} \label{prop:more_orthogonality}
    Let $\le$ be a total order on $W$ satisfying the following properties:
    \begin{itemize}
        \item if $v\le w$ then $l(v) \le l(w)$,
        \item if $v\le w$ and $l(v) = l(w)$ then $\| e_v\| \le \|e_w\|$.
    \end{itemize}
    Then in the notation of Definition~\ref{def:Xv} and Theorem~\ref{thm:SOD_SvdK} the following holds.
    \begin{enumerate}
        \item For $v\in W^P$ the isomorphism class of the object $\hat{X}_v\in \D^b(\rep P)$ does not depend on the choice of the total order (satisfying the above properties).
        
        \item
        For $1\le i,j\le n$ such that $l(v_i)=l(v_j)$ and $\|e_{v_i}\|= \|e_{v_j}\|$ the subcategories $\Phi_i(\D^b(\rep G))$, $\Phi_j(\D^b(\rep G))\subseteq \D^b(\rep P)$ are pairwise orthogonal.
    \end{enumerate}
\end{proposition}
\begin{proof}
    The first claim follows from Corollary~\ref{cor:Xv_independent}, the second one from Corollary~\ref{cor:orthogonal}.
\end{proof}

\subsection{Quasi-split simply connected case}
\label{sec:qs_sc}
In this section we generalize the results of \cite{SvdK24} to the case of a simply connected quasi-split semisimple group constructing a full $G$-relative \'etale-exceptional collection in $\D^b(\rep P)$.

Recall that a semisimple group $G$ over a field $k$ is \textit{quasi-split} if there exists a Borel subgroup $B\leqslant G$ \cite[Definition~17.67]{Mi17}. In this section $G$ is a quasi-split simply connected semisimple group over $k$ and $T\leqslant B\leqslant P\leqslant G$ is a maximal (possibly non-split) torus, a Borel subgroup, and a parabolic subgroup respectively.

\begin{definition}
The \textit{splitting field of $G$} is the finite Galois field extension $L/k$ corresponding to the kernel of the action of the absolute Galois group $\Gal(\ksep/k)$ on the lattice $\Hom(T_{k^{\mathrm{sep}}},(\mathbb{G}_m)_{k^{\mathrm{sep}}})$. Note that since all Borel subgroups are conjugate over $k$ \cite[Theorem~25.8]{Mi17}, the isomorphism class of $T$ does not depend on the choice of $T\leqslant B\leqslant G$, hence $L/k$ depends (up to an isomorphism) only on $G$ and not on $T$. The extension $L/k$ is the minimal field extension of $k$ such that $G_L$ is split. We denote $\Gamma:=\Gal(L/k)$ the Galois group of the splitting field of $G$ and put $X^*(T):=\Hom(T_{L},\mathbb{G}_m)$ viewed as a right $\Gamma$-module (this action is compatible with the $\Gamma$-action on the representation categories introduced in Definition~\ref{def:gamma_rep}). Since $G_L$ is split, over $L$ we have all the combinatorial gadgets considered in Section~\ref{sec:split_sc}, in particular, we have the set of simple roots $\Pi\subseteq X^*(T)$, the chamber of dominant weights $X^*(T)_+\subseteq X^*(T)$ generated by fundamental weights $\varpi_{\alpha}\in X^*(T)_+$, $\alpha\in \Pi$, and as in Section~\ref{sec:split_sc}, $B_L$ corresponds to the negative roots. We denote $W:=W(G):=W(G_{L})$ and $W_P:=W(P_{L})$ the Weyl groups of $G_{L}$ and $P_{L}$ respectively, and $W^P$ the set of minimal (in Bruhat order) representatives for cosets $W/W_P$ \cite[\S~1.10]{Hu92}. Since $B$ is defined over $k$, the group $\Gamma$ permutes simple roots and, furthermore, this permutation action gives an embedding $\Gamma\leqslant \mathrm{Aut}(\Delta)$ of $\Gamma$ into the automorphism group of the Dynkin diagram $\Delta:=\Delta(G_{L})$. All the combinatorial data comes canonically equipped with the right $\Gamma$-action in a compatible way, in particular, $\Gamma$ acts on $X^*(T)_+$, permutes fundamental weights, and acts by automorphisms on $W$ and $W_P$ in such a way that $(w\lambda)^{\gamma} = w^{\gamma} \lambda^{\gamma}$ for $w\in W$, $\gamma\in \Gamma$, $\lambda\in X^*(T)$. Furthermore, it is straightforward to check that $\Gamma$-action on $W$ permutes the elements of $W^P$. Since $\Gamma$ permutes simple roots, it also permutes simple reflections in $W$, in particular, $l(w)=l(w^\gamma)$ for every $w\in W$ and $\gamma\in \Gamma$. 
\end{definition}

\begin{lemma} \label{lem:weight_defined}
    For $v\in W$ and $\gamma\in \Gamma$ we have $(e_v)^{\gamma} = e_{v^{\gamma}}$, where $e_v\in X^*(T)$ is the Steinberg weight associated to $v$ as in Definition~\ref{def:Steinberg}. In particular, $\operatorname{Stab}_\Gamma v = \operatorname{Stab}_\Gamma e_v$. 
\end{lemma}
\begin{proof}

     Recall that 
     \[
     e_v=v^{-1}\sum_{\substack{\alpha\in \Pi\\ v^{-1}\alpha<0}} \varpi_\alpha.
     \]
     Thus
     \[
     (e_v)^{\gamma}=(v^{-1})^{\gamma}\sum_{\substack{\alpha\in \Pi\\ v^{-1}\alpha<0}} (\varpi_\alpha)^{\gamma} = (v^{\gamma})^{-1}\sum_{\substack{\beta\in \Pi\\ v^{-1}(\beta^{(\gamma^{-1})})<0}} \varpi_\beta,
     \]
     where for the last equality we used $(\varpi_\alpha)^{\gamma} = \varpi_{\alpha^{\gamma}}$ and substituted $\beta:=\alpha^{\gamma}$.     
     Note that $v^{-1}(\beta^{(\gamma^{-1})})=((v^{\gamma})^{-1}\beta)^{\gamma^{-1}}$, and since $\gamma^{-1}$ preserves positive and negative roots, $v^{-1}(\beta^{(\gamma^{-1})})<0$ if and only if $(v^{\gamma})^{-1}\beta<0$. The claim follows.
\end{proof}

\begin{remark}
    It follows from the above Lemma that for $v\in W$ we have $\operatorname{Stab}_{\Gamma} v\leqslant \operatorname{Stab}_\Gamma ve_v$ for the dominant weight $ve_v$, and in general the inclusion can be strict.
\end{remark}

\begin{definition} \label{def:kv}
    For an intermediate field extension $k\subseteq K\subseteq L$ and $w\in W$ we denote $\Gamma_K:=\Gal(L/K)\leqslant \Gamma$ and $K\subseteq K(w)\subseteq L$ the intermediate extension corresponding to $\operatorname{Stab}_{\Gamma_K} w\leqslant \Gamma_K$. 
\end{definition}

\begin{definition}
    Let $k\subseteq K\subseteq L$ be an intermediate field extension, $v\in W^P$, and $w_0\in W$ be the longest element. Then $v$ is defined over $K(v)$ in the sense that there exists a rational point on $(\rfaktor{N_G(T)}{T})_{K(v)}$ that maps to $v\in W=(\rfaktor{N_G(T)}{T})_{L}$ under the extension of scalars. Since $w_0$ is stabilized by $\Gamma$, $w_0v$ is also defined over $K(v)$. Hence for $w=v$ and $w=w_0v$ the Schubert variety $\rfaktor{\overline{B_Lw^{-1}B_L}}{B_L}$ and its boundary $\left(\rfaktor{\overline{B_Lw^{-1}B_L}}{B_L}\right)\setminus \left(\rfaktor{B_Lw^{-1}B_L}{B_L}\right)$ are defined over $K(v)$, i.e. these schemes are obtained by base change from some $Z_{w^{-1}}, \partial Z_{w^{-1}}\subseteq (\rfaktor{G}{B})_{K(v)}$. Furthermore, by Lemma~\ref{lem:weight_defined} the characters $-w_0ve_v\in X^*(T)$ and $ve_v\in X^*(T)$ are defined over $K(v)$ giving rise to $1$-dimensional $K(v)$-representations $K(v)_{-w_0ve_v}$ and $K(v)_{ve_v}$ of $B$ with characters $-w_0ve_v$ and $ve_v$ respectively, which become $L_{-w_0ve_v}$ and $L_{ve_v}$ after the extension of scalars. We define the dual Joseph module $\mathcal{P}_{v,K(v)} \in \lrep{K(v)} B$ and the relative Schubert module $\mathcal{Q}_{v,K(v)} \in \lrep{K(v)} B$ by the same formulas as in the split case,
    \begin{equation*}
    \begin{gathered}
        \mathcal{P}_{v,K(v)}:=\mc{P}_{K(v)}(-e_v):=\mathrm{H}^0\left( Z_{w^{-1}},\mathscr{L}(K(v)_{-we_v})|_{Z_{w^{-1}}}\right),\quad w:=w_0v,\\
        \mathcal{Q}_{v,K(v)}:=\mc{Q}_{K(v)}(e_v):=\ker \left(\mathrm{H}^0( Z_{v^{-1}},\mathscr{L}(K(v)_{ve_v})|_{Z_{v^{-1}}}) \to \mathrm{H}^0(\partial Z_{v^{-1}},\mathscr{L}(K(v)_{ve_v})|_{\partial Z_{v^{-1}}})\right).
    \end{gathered}
    \end{equation*}
    It is clear that for a tower $k\subseteq K\subseteq K'\subseteq L$ we have canonical isomorphisms
    \[
    \mathcal{P}_{v,K(v)}\otimes_{K(v)} K'(v)\cong \mathcal{P}_{v,K'(v)},\quad
    \mathcal{Q}_{v,K(v)} \otimes_{K(v)} K'(v)\cong \mathcal{Q}_{v,K'(v)},
    \]
    and, furthermore,
    \[
    \mathcal{P}_{v,L}\cong \mathcal{P}_v,\quad
    \mathcal{Q}_{v,L}\cong \mathcal{Q}_v,
    \]
    where on the right-hand side we have the corresponding modules for the split group $G_L$ as in Definition~\ref{def:Pv}.
    
    We put
    \[
     \hat{\mathcal{P}}_{v,K(v)}:=\ind_{B}^{P} \mathcal{P}_{v,K(v)},\quad   \hat{\mathcal{Q}}_{v,K(v)}:=\ind_{B}^{P} \mathcal{Q}_{v,K(v)}.
    \]
    Recall that for $v\in W^P$ the counit homomorphism $\res_{B}^{P} \ind_{B}^{P} \mathcal{P}_{v,L} \to \mathcal{P}_{v,L}$ is an isomorphism by \cite[Lemma~14.15]{SvdK24}, and since extension of scalars is faithful, the counit homomorphism
    \[
    \res_{B}^{P} \ind_{B}^{P} \mathcal{P}_{v,K(v)} \to \mathcal{P}_{v,K(v)}
    \]
    is an isomorphism as well, that is $\mathcal{P}_{v,K(v)}\cong \res_{B}^{P} \hat{\mathcal{P}}_{v,K(v)}$ as in the split case. 

\end{definition}

\begin{definition}
    A \textit{transversal} $\Theta\subseteq \mc{S}$ for a right action of a group $\Gamma$ on a set $\mc{S}$ is a collection of representatives of the orbits $\rfaktor{\mc{S}}{\Gamma}$.
\end{definition}

\begin{lemma}\label{lem:invariant}
    Let $R$ be a commutative ring equipped with a right action of a finite group $\Gamma$, let $S\subseteq R$ be a $\Gamma$-invariant subring and $\mc{N}\subseteq R$ be a $\Gamma$-invariant subset such that $R=\bigoplus_{x\in \mc{N}} S\cdot x$, let $\Theta\subseteq \mc{N}$ be a transversal for the $\Gamma$-action and denote $\Gamma_{x}:= \operatorname{Stab}_{\Gamma} x$ for $x\in \Theta$. Then
    \[
    R^\Gamma = \bigoplus_{x\in \Theta}\left\{\left.\sum_{\gamma\in \lfaktor{\Gamma_{x}}{\Gamma}}  s^\gamma x^\gamma\,\right|\,  {s\in S^{\Gamma_{x}}} \right\}.
    \]
    Here $R^\Gamma$ and $S^{\Gamma_{x}}$ denote the respective subrings of invariant elements.
\end{lemma}
\begin{proof}
    Let $x\in R^\Gamma$. Then $x=\sum_{x\in \mc{N}} s_x x$ for some $s_x\in S$. Since $\mc{N}$ is a $\Gamma$-invariant basis, $s_{x}^\gamma = s_{x^{\gamma}}$ for every  $\gamma\in\Gamma$ and $x\in\mc{N}$. The claim follows.
\end{proof}


\begin{theorem}\label{thm:generation}
    Let $\Theta\subseteq W^P$ be a transversal for the $\Gamma$-action and let $N_{v}\in \D^b(\lrep{k(v)} P)$, $v\in \Theta$, be a collection of objects. Suppose that
    \[
    \D^b(\lrep{L} P)= \hull\left( \left\{M\otimes_L (N_{v}\otimes_{k(v)} L)^\gamma \right\}_{M\in \lrep{L} G,\, v\in \Theta,\, \gamma\in\Gamma}\right).
    \]
    Then
    \begin{equation} \label{eq:gen_P}
        \D^b(\rep P)= \hull\left( \left\{(r_{k(v)})_*(M\otimes_{k(v)} N_v) \right\}_{M\in \lrep{k(v)} G,\,v\in \Theta}\right).
    \end{equation}
\end{theorem}
\begin{proof}
    Let $\D$ be the right-hand side of~\refbr{eq:gen_P}. We would like to appeal to Thomason's classification of dense triangulated subcategories \cite[Theorem~2.1]{Th97}, for which we need to show that $\D$ is dense in $\D^b(\rep P)$ and that $K_0(\D)= K_0(\D^b(\rep P))$.
    
    For $N\in \D^b(\rep P)$ we have $(r_L)_*(N\otimes_k L)\cong N^{\oplus [L:k]}$, thus
    \[
    \left((r_L)_*(\D^b(\lrep{L} P))\right)^{\oplus} =\D^b(\rep P),
    \]
    where the superscript $\oplus$ denotes the thick closure. Furthermore, for $M\in \lrep{L} G$, $v\in\Theta$ and $\gamma\in \Gamma$ we have
    \begin{multline*}
    (r_L)_*(M\otimes_L (N_v\otimes_{k(v)} L)^\gamma)\cong (r_L)_*((M^{\gamma^{-1}}\otimes_L (N_v\otimes_{k(v)} L))^\gamma)\cong \\ 
    \cong (r_L)_*(M^{\gamma^{-1}}\otimes_L (N_v\otimes_{k(v)} L))\cong (r_{k(v)})_*\left((r_{L/k(v)})_*(M^{\gamma^{-1}})\otimes_{k(v)} N_v\right).
    \end{multline*}
    Combining the above we obtain $\D^b(\rep P)= \D^\oplus$, i.e. $\D$ is dense in $\D^b(\rep P)$.

    Consider the following representation rings. 
    \[
    R(P):=K_0(\rep P),\quad R_L(P):=K_0(\lrep{L} P),\quad R_L(G):=K_0(\lrep{L} G),\quad R_{k(v)}(G):=K_0(\lrep{k(v)} G).
    \]
    For $w\in W^P$ let $[N_{w,L}]:=[(N_v\otimes_{k(v)} L)^\gamma] \in R_L(P)$ for $v\in \Theta$ and $\gamma\in \Gamma$ such that $w=v^{{\gamma}}$, note that this isomorphism class does not depend on the particular choice of $\gamma$. By the assumption the classes $\{[N_{w,L}]\}_{w\in W^P}$ generate $R_L(P)$ as an $R_L(G)$-module, and since $R_L(P)$ is a free $R_L(G)$-module of rank $\# W^P$ by a theorem of Steinberg \cite{St75}, the classes $\{[N_{w,L}]\}_{w\in W^P}$ are a basis of $R_L(P)$ over $R_L(G)$. Furthermore, the basis is $\Gamma$-invariant by the construction. By Tits' classification of representations of quasi-split groups (see also \cite[Fact on p.~576]{Pan94}) extension of scalars induces isomorphisms $R(P)\cong R_L(P)^{\Gamma}$ and $R_{k(v)}(G)\cong R_L(G)^{\Gamma_v}$ with $\Gamma_v=\operatorname{Stab}_{\Gamma} v$. Then Lemma~\ref{lem:invariant} yields
    \[
    R(P)\cong R_L(P)^{\Gamma} = \bigoplus_{v\in \Theta}\left\{\sum_{\gamma\in \lfaktor{\Gamma_{v}}{\Gamma}}  [M\otimes_{k(v)} L]^\gamma \cdot [N_{v^\gamma,L}]\right\}_{[M]\in R_{k(v)}(G)}.
    \]
    For $v\in \Theta$ and $M\in \lrep{k(v)} G$ we have
    \[
    \left((r_{k(v)})_*(M\otimes_{k(v)} N_{v})\right) \otimes_k L \cong \bigoplus_{\gamma \in \lfaktor{\Gamma_{v}}{\Gamma}} (M\otimes_{k(v)} L)^\gamma \otimes_L N_{v^\gamma,L}.
    \]
    Thus $K_0(\D^b(\rep P))=K_0(\D)$ and the claim of the Theorem follows from \cite[Theorem~2.1]{Th97}.
    \end{proof}

\begin{definition} \label{def:PQ_qs}
    Let $k\subseteq K\subseteq L$ be an intermediate field extension, $\Gamma_K:=\Gal(L/K)\leqslant \Gamma$ and $I\subseteq W^P$ be a $\Gamma_K$-invariant initial segment in the Bruhat order. Denote
    \[
        \hat{\mc{P}}_{I,K} :=\hull \left( \left\{(r_{K(v)/K})_*\left(M\otimes_{K(v)} \hat{\mc{P}}^*_{v,K(v)}\right) \right\}_{M\in \lrep{K(v)} G,\,v\in I}\right)\subseteq \D^b(\lrep{K} P).
    \]
    Similarly, put $\bar{I}:=W^P\setminus I$ and denote
    \[
        \hat{\mc{Q}}_{\bar{I},K} :=\hull \left( \left\{(r_{K(v)/K})_*\left(M\otimes_{K(v)} \hat{\mc{Q}}_{v,K(v)}\right) \right\}_{M\in \lrep{K(v)} G,\,v\in \bar{ I}}\right)\subseteq \D^b(\lrep{K} P).
    \]
\end{definition}

\begin{lemma}\label{lem:PQ_ext_scalars}
    In the notation of Definitions~\ref{def:Xv} and~\ref{def:PQ_qs} one has
    \[
    \hat{\mc{P}}_{\bar{I},K}\otimes_K L \subseteq \hat{\mc{P}}_{\bar{I}},\quad \hat{\mc{Q}}_{\bar{I},K}\otimes_K L \subseteq \hat{\mc{Q}}_{\bar{I}}.
    \]
\end{lemma}
\begin{proof}
    Let $\Gamma_{K(v)}:=\Gal(L/K(v))$, then 
    \[
    (r_{K(v)/K})_*\left(M\otimes_{K(v)} \hat{P}^*_{v,K(v)}\right) \otimes_K L
    \cong \bigoplus_{\gamma \in \lfaktor{\Gamma_{K(v)}}{\Gamma}} (M\otimes_{K(v)} L)^\gamma \otimes_{L} (\hat{\mc{P}}^*_{v,K(v)} \otimes_{K(v)} L)^{\gamma}.
    \]
    Since $\hat{\mc{P}}^*_{v,K(v)} \otimes_{K(v)} L \cong \hat{\mc{P}}^*_{v}$ and $(\hat{\mc{P}}_{v})^{\gamma} \cong \hat{\mc{P}}_{v^\gamma}$ the first claimed inclusion follows. The second one is obtained in a similar way.
\end{proof}

\begin{proposition} \label{prop:PQ_qs}
    In the notation of Definition~\ref{def:PQ_qs} there is a semiorthogonal decomposition
    \begin{equation} \label{eq:sod_qs}
        \D^b(\lrep{K} P) = \left\langle \hat{\mc{Q}}_{\bar{I},K}, \hat{\mc{P}}_{I,K} \right\rangle.
    \end{equation}
\end{proposition}
\begin{proof}
    Semiorthogonality follows from the semiorthogonality property of the decomposition~\refbr{eq:SOD_PQ} and Lemma~\ref{lem:PQ_ext_scalars} since extension of scalars to $L$ is a faithful functor. Generation follows from Theorem~\ref{thm:generation} and the generation part of the semiorthogonal decomposition~\refbr{eq:SOD_PQ}.
\end{proof}

\begin{definition} \label{def:compaible_action}
    We say that a total order $\le$ on $W$ is \textit{compatible with $\Gamma$-action} if for every $w,v\in W$ and $\gamma\in \Gamma$ such that $w\le v\le w^\gamma$ there exists $\gamma'\in \Gamma$ such that $v=w^{\gamma'}$. Note that one may always extend the Bruhat order $\lebru$ on $W$ to a total order in a compatible with $\Gamma$-action way, for example, one may first order the elements with respect to the length function $l(-)$, then order the $\Gamma$-orbits consisting of the elements of the same length, and then arbitrary order the elements inside every $\Gamma$-orbit. 
\end{definition}

\begin{lemma} \label{lem:X_Gamma_action}
    Let $\le$ be a total order on $W$ extending the Bruhat order $\lebru$ and compatible with $\Gamma$-action. Then for $v\in W^P$, $\gamma\in\Gamma$ and $\hat{X}_v, \hat{X}_{v^\gamma}\in \D^b(\lrep{L} P)$ defined as in Definition~\ref{def:Xv} one has
    \[
    (\hat{X}_v)^\gamma\cong \hat{X}_{v^\gamma}.
    \]
\end{lemma}
\begin{proof}
    Let $\le'$ be the total order on $W$ with $w_1 \le' w_2$ if and only if $w_1^{\gamma^{-1}} \le w_2^{\gamma^{-1}}$, and let $\hat{X}'_{v^\gamma}$ be the object corresponding to $v^\gamma$ and total order $\le'$. We have $(\hat{\mc{P}}_v)^\gamma \cong \hat{\mc{P}}_{v^\gamma}$ and  for $w\in W$ we have $(\hat{\mc{Q}}_w)^\gamma \cong \hat{\mc{Q}}_{w^\gamma}$, thus $(\hat{\mc{Q}}_{\ge v})^\gamma = \hat{\mc{Q}}_{\ge' v^{\gamma}}$ and hence
    \[
    (\hat{X}_v)^\gamma \cong \hat{X}'_{v^\gamma}.
    \]
    
    In the notation of Definition~\ref{def:ant} we have $\left\{ v^{\gamma'}\,|\, \gamma'\in \Gamma \right\} \cap J_{v^\gamma} =\{v^\gamma\}$
    since $\|e_{v^\gamma} \| = \|e_{v^{\gamma'}} \|$ and $l(v^\gamma)=l(v^{\gamma'})$ for $\gamma'\in \Gamma$. Hence $I_{\le v^\gamma}\cap J_{v^\gamma}=I_{\le' v^\gamma}\cap J_{v^\gamma}$ and Corollary~\ref{cor:Xv_independent} yields
    \[
    \hat{X}_{v^\gamma} \cong \hat{X}'_{v^\gamma}.
    \]
    The claim follows.
\end{proof}
    
\begin{definition} \label{def:Xv_qs}
    Let $\le$ be a total order on $W$ extending the Bruhat order $\lebru$ and compatible with $\Gamma$-action. For $v\in W^P$ we put 
    \[
    I^{\Gamma}_v := \bigcap_{\gamma\in \Gamma} \left\{ w\in W^P\,|\, w< v^\gamma\right\},\quad
    \bar{I}^{\Gamma}_v := \left\{ w\in W^P\,|\, v\le w\right\} \cup \left\{v^\gamma\,|\, \gamma\in \Gamma\right\},
    \]
    and denote
    \[
        \hat{X}_{v,k(v)}:=(j_{\bar{I}^{\Gamma}_v,\, k(v)})_*j_{\bar{I}^{\Gamma}_v,\,k(v)}^* \hat{\mc{P}}_{v,\,k(v)}^*\in \D^b(\rep P_{k(v)}),
    \]
    where 
    \[
    j_{\bar{I}^{\Gamma}_v,\,{k(v)}}^*\colon \D^b(\rep P_{k(v)})\leftrightarrows \hat{\mc{Q}}_{\bar{I}^{\Gamma}_v,\,{k(v)}}\colon (j_{\bar{I}^{\Gamma}_v,\, k(v)})_*
    \]
    are the embedding and its left adjoint that exists by Proposition~\ref{prop:PQ_qs}.
\end{definition}

\begin{proposition} \label{prop:pushpull_X}
    Let $\le$ be a total order on $W$ extending the Bruhat order $\lebru$ and compatible with $\Gamma$-action. Then the following holds.
    \begin{enumerate}
        \item For $v\in W^P$ and $\gamma\in\Gamma$ we have $(\hat{X}_{v,k(v)} \otimes_{k(v)} L)^\gamma \cong \hat{X}_{v^\gamma}$.
        \item For $v\in W^P$ and $M\in \lrep{k(v)} G$ we have
    \[
        \left((r_{k(v)})_*\left(M \otimes_{k(v)} \hat{X}_{v,k(v)}\right)\right) \otimes_{k} L \cong \bigoplus_{\gamma \in \lfaktor{\Gamma_v}{\Gamma}} \left(M \otimes_{k(v)} L\right)^\gamma \otimes \hat{X}_{v^\gamma}.
    \]
    \end{enumerate} 
    Here $\hat{X}_{v^\gamma}$ is as in Definition~\ref{def:Xv} and $\Gamma_v=\Gal(L/k(v))$.
\end{proposition}
\begin{proof}
    First suppose that $v\le v^{\gamma}$ for every $\gamma\in \Gamma$. Then in the notation of Definition~\ref{def:Xv_qs} we have
    \[
    \bar{I}^{\Gamma}_{v} = \left\{w\in W^P\,|\, v\le w\right\} \quad I^{\Gamma}_{v}= \left\{w\in W^P\,|\, w< v\right\}.
    \]
    We have $\hat{\mc{P}}_{v,k(v)} \otimes_{k(v)} L \cong \hat{\mc{P}}_{v}$, and, by Lemma~\ref{lem:PQ_ext_scalars}, 
    \[
    \hat{\mc{Q}}_{\bar{I}^{\Gamma}_{v},\,k(v)}\otimes_{k(v)} L \subseteq \hat{\mc{Q}}_{\ge v},\quad     \hat{\mc{P}}_{I^{\Gamma}_{v},\,k(v)}\otimes_{k(v)} L \subseteq \hat{\mc{P}}_{< v},
    \]
    thus $\hat{X}_{v,k(v)} \otimes_{k(v)} L \cong \hat{X}_{v}$, so we get the first claim for $v$ being the minimal element in its $\Gamma$-orbit. The first claim for a general $v$ follows from the above, isomorphism
    \[
        (\hat{X}_{v,k(v)} \otimes_{k(v)} L)^{\gamma} \cong \hat{X}_{v^\gamma,k(v^\gamma)} \otimes_{k(v^\gamma)} L
    \]
    and Lemma~\ref{lem:X_Gamma_action}.

    The second claim follows from the first one and Lemma~\ref{lem:pushpull}.
\end{proof}

\begin{theorem} \label{thm:SOD_rep_qs}
    Let $\le$ be a total order on $W$ extending the Bruhat order $\lebru$ and compatible with the $\Gamma$-action, and $\{v_1,v_2,\hdots, v_m\}\subseteq W^P$ be a transversal for the $\Gamma$-action with elements enumerated in such a way that $v_m<v_{m-1}<\hdots<v_1$. Denote 
    \[
    \hat{X}_i:=\hat{X}_{v_i,k(v_i)}\in \D^b(\lrep{k(v_i)} P)
    \]
    in the notation of Definition~\ref{def:Xv_qs}.  Then $\{\hat{X}_i\}_{1\le i\le m}$ is a full $G$-relative \'etale-exceptional collection in $\D^b(\rep P)$, the functors 
    \[
    \Phi_i\colon \D^b(\lrep{k(v_i)} G) \to \D^b(\rep P), \quad M\mapsto (r_{k(v_i)})_*(M\otimes_{k(v_i)} \hat{X}_i),
    \]
    are fully faithful and there is a semiorthogonal decomposition
        \[
        \D^b(\rep P) = \langle \Phi_1(\D^b(\lrep{k(v_1)} G)), \Phi_2(\D^b(\lrep{k(v_2)} G)), \hdots, \Phi_{m}(\D^b(\lrep{k(v_m)} G))\rangle.
        \] 
\end{theorem}
\begin{proof}
    For $1\le i\le m$ denote $(r_i)_*:= (r_{k(v_i)})_*$ and $\Gamma_i:=\Gal(L/k(v_i))\leqslant \Gamma$. For $1\le i,j\le m$ by Proposition~\ref{prop:pushpull_X} we have
    \begin{multline*}
    \left(\ind^G_P \iHom_{\D^b(\rep P)}(\hat{X}_i,\hat{X}_j)\right) \otimes L
    \cong
    \ind^{G}_{P} \iHom_{\D^b(\lrep{L} P)}(\hat{X}_i\otimes L,\hat{X}_j\otimes L) 
    \cong 
    \\
    \cong \bigoplus_{\gamma\in \lfaktor{\Gamma_i}{\Gamma},\, \gamma'\in \lfaktor{\Gamma_j}{\Gamma}} \ind^{G}_{P} \iHom_{\D^b(\lrep{L} P)}(\hat{X}_{v_i^\gamma}, \hat{X}_{v_j^{\gamma'}}),
    \end{multline*}
    where $\hat{X}_{v_i^\gamma}, \hat{X}_{v_j^{\gamma'}}\in \D^b(\lrep{L} P)$ are as in Definition~\ref{def:Xv}. 

    Recall that $\hat{X}_{v_i^\gamma}$ form a $G$-relative exceptional collection by Theorem~\ref{thm:SOD_SvdK}, thus applying Corollary~\ref{cor:orthogonal} we obtain
    \[
        \ind^{G}_{P} \iHom_{\D^b(\lrep{L} P)}(\hat{X}_{v_i^\gamma}, \hat{X}_{v_i^{\gamma'}})=
        \begin{cases}
            L,&\gamma=\gamma',\\
            0,&\gamma\neq \gamma'.
        \end{cases}
    \]
    Hence
    \[
    \left(\ind^G_P \iHom_{\D^b(\rep P)}(\hat{X}_i,\hat{X}_i)\right) \otimes L\cong \bigoplus_{\gamma\in \lfaktor{\Gamma_i}{\Gamma}} L \cong k(v_i) \otimes_k L,
    \]
    and it is straightforward to check that this isomorphism coincides with $\theta_{\hat{X}_i}\otimes L$ for the canonical homomorphism
    \[
    \theta_{\hat{X}_i}\colon k(v_i) \to \left(\ind^G_P \iHom_{\D^b(\rep P)}(\hat{X}_i,\hat{X}_i)\right).
    \]
    Since extension of scalars to $L$ is conservative then $\theta_{\hat{X}_i}$ is an isomorphism, hence $\hat{X}_i$ is $k(v_i)$-exceptional.
    
    For $i>j$ we have $v_i^\gamma <v_j^{\gamma'}$ since $\le$ is compatible with $\Gamma$-action. Thus $\ind^{G}_{P}\iHom_{\D^b(\lrep{L} P)}(\hat{X}_{v_i^\gamma}, \hat{X}_{v_j^{\gamma'}})=0$ by Theorem~\ref{thm:SOD_SvdK}, yielding 
    \[
    \left(\ind^G_P \iHom_{\D^b(\rep P)}(\hat{X}_i,\hat{X}_j)\right) \otimes L=0.
    \]
    Since the extension of scalars to $L$ is conservative, we find that $\{\hat{X}_i\}_{1\le i\le m}$ form a $G$-relative \'etale-exceptional collection.

    For $1\le i \le m$ and $\gamma\in \Gamma$ by Proposition~\ref{prop:pushpull_X} and Lemma~\ref{lem:X_Gamma_action} we have
    \[
    \hat{X}_i\otimes_{k(v_i)} L \cong \hat{X}_{v_i},\quad (\hat{X}_i\otimes_{k(v_i)} L)^{\gamma}\cong \hat{X}_{v_i^\gamma}.
    \]
    Thus Theorem~\ref{thm:SOD_SvdK} allows us to apply Theorem~\ref{thm:generation} obtaining the fullness property of the collection.
    
    Hence $\{\hat{X}_i\}_{1\le i\le m}$ is a full $G$-relative \'etale-exceptional collection in $\D^b(\rep P)$.

    The remaining claims follow from Proposition~\ref{prop:SOD_from_collection_rep}.
\end{proof}

    

\subsection{Representations of isogenous groups}
\label{sec:representations_isogenous}
In this section we study the relation between the derived categories of representations of isogenous affine algebraic groups showing that the derived category of representations of the covering group decomposes into a direct sum of derived categories of representations over some separable algebras for the quotient group.

Recall that a (not necessarily reduced) affine algebraic group scheme $Z$ over a field $k$ is \textit{of multiplicative type} if $Z_{\ksep}$ is diagonalizable \cite[Chapter~12]{Mi17}. A field extension $L/k$ \textit{splits} $Z$ if $Z_L$ is diagonalizable. The category of group schemes of multiplicative type over a field $k$ is dual to the category of finitely generated $\mathbb{Z}$-modules equipped with a continuous right action of $\mathrm{Gal}(\ksep/k)$ with the contravariant equivalence given by 
\[
Z\mapsto \Hom(Z_{\ksep},(\Gm_m)_{\ksep}).
\]
In this section $\rho \colon \tilde{G}\to G$ is a surjective homomorphism of affine algebraic groups over $k$ with $Z:=\ker \rho\leqslant \tilde{G}$ being a central subgroup scheme of multiplicative type, and $L/k$ is a finite Galois field extension splitting $Z$. We denote $\Gamma:=\Gal(L/k)$.


\begin{definition} \label{def:blocks}
Recall that the category of representations $\rep Z$ is semisimple, and the isomorphism classes of simple objects in $\rep Z$ are naturally parametrized by the Galois orbits $\rfaktor{X^*(Z)}{\Gamma}$ \cite[Theorem~12.30]{Mi17} with $X^*(Z):=\Hom(Z_L,(\Gm_m)_L)$. 

Let $M\in \rep \tilde{G}$, then $\res^{\tilde{G}}_Z (M\otimes L)$ decomposes into a direct sum of its eigenspaces $(\res^{\tilde{G}}_Z (M\otimes L))_\lambda$, $\lambda\in X^*(Z)$, which descends to the direct sum decomposition
        \[
            \res^{\tilde{G}}_Z M = \bigoplus_{\theta \in \rfaktor{X^*(Z)}{\Gamma}} (\res^{\tilde{G}}_Z M)_{\theta},
        \]
    such that $(\res^{\tilde{G}}_Z M)_{\theta}\otimes L = \bigoplus_{\lambda \in \theta} (\res^{\tilde{G}}_Z (M\otimes L))_{\lambda}$. Since $Z$ is central in $\tilde{G}$, the eigenspaces are stable under the action of $\tilde{G}$, yielding a canonical decomposition
    \[
    M=\bigoplus_{\theta \in \rfaktor{X^*(Z)}{\Gamma}} M_{\theta}
    \]
    in $\rep \tilde{G}$. There are no nontrivial homomorphisms between different eigenspaces, thus we obtain an orthogonal decomposition of abelian categories
    \begin{equation} \label{eq:rep_blocks}
    \rep \tilde{G} = \bigoplus_{\theta \in \rfaktor{X^*(Z)}{\Gamma}} (\rep \tilde{G})_\theta,
    \end{equation}
    where $(\rep \tilde{G})_\theta$ denotes the full subcategory of $\rep \tilde{G}$ consisting of representations $M$ with $M=M_\theta$, i.e. such that $Z$ acts on $M\otimes L$ through the characters $\lambda\in \theta$. Consequently, we have an orthogonal decomposition
    \begin{equation} \label{eq:Db_rep_blocks}
    \D^b(\rep \tilde{G}) = \bigoplus_{\theta \in \rfaktor{X^*(Z)}{\Gamma}} \D^b(\rep \tilde{G})_\theta
    \end{equation}
    with $\D^b(\rep \tilde{G})_\theta:=\D^b((\rep \tilde{G})_\theta)$. These decompositions are compatible with the tensor product of representations, 
    \[
    \begin{gathered}
    (\rep \tilde{G})_\theta\otimes (\rep \tilde{G})_{\theta'}\subseteq \bigoplus_{\theta''\subseteq \theta+\theta'}(\rep \tilde{G})_{\theta''},\\
    \D^b(\rep \tilde{G})_\theta\otimes \D^b(\rep \tilde{G})_{\theta'}\subseteq \bigoplus_{\theta''\subseteq \theta+\theta'} \D^b(\rep \tilde{G})_{\theta''}.
    \end{gathered}
    \]
    Note that in general for $\theta,\theta'\in \rfaktor{X^*(Z)}{\Gamma}$ the set of characters $\theta+\theta'$ is a union of several $\Gamma$-orbits, but if $\theta'=\{\lambda\}$ for $\lambda\in X^*(Z)^{\Gamma}$ then $\theta+\theta'$ is a single orbit. 
\end{definition}

\begin{definition} \label{def:Azumaya_equiv}
    For $M\in \rep \tilde{G}$ we have $\iEnd M\in\Alg_{\tilde{G}} k$ which is given by the endomorphism algebra $\End M$ of the underlying vector space equipped with the conjugation action of $\tilde{G}$. Tautologically, we have $M\in \lrep{\left(\iEnd M\right)} \tilde{G}$. Suppose that $M=M_{\{\lambda\}}$ for some $\lambda\in X^*(Z)^{\Gamma}$. Then the action of $Z$ on $\iEnd M$ is trivial and the action of $\tilde{G}$ on $\iEnd M$ descends to an action of $G$ giving rise to an equivariant algebra $A_M\in \Alg_G k$ satisfying $\res^{G}_{\tilde{G}} A_M \cong \iEnd M$. Forgetting the equivariant structure we have $A_M\cong M_m(k)$ with $m=\dim_k M$, so $A_M\in \CSAlg_G k$.
\end{definition}

\begin{lemma} \label{lem:Azumaya_equiv}
    Let $\lambda\in X^*(Z)^{\Gamma}$ and $M,N\in(\rep \tilde{G})_{\{\lambda\}}$. Then $Z$ acts trivially on $\iHom(N,M)$ and we denote $A_{N,M}\in 
    \lrep{\left(A_M\otimes A_N^{\op}\right)} G$ the corresponding representation satisfying $\res^G_{\tilde{G}} A_{N,M} \cong\Hom (N,M)$. Then
        \[
        A_N\otimes \iEnd A_{N,M} \cong A_M \otimes \iEnd A_{N}
        \]
        in $\Alg_G k$ and the functor
        \[
        A_{N,M} \otimes_{A_N} - \colon \lrep{A_N} G \to \lrep{A_M} G
        \]
        is an equivalence.
\end{lemma}
\begin{proof}
    The functor $\res^G_{\tilde{G}}\colon \rep G \to \rep \tilde{G}$ is fully faithful, so in order to obtain the isomorphism we may assume that $G=\tilde{G}$, so $N,M\in \rep G$ and $A_{N,M}\cong N^*\otimes M$. Then
    \[
    A_N\otimes \iEnd A_{N,M} \cong \iEnd (N\otimes N^*\otimes M)\cong A_M \otimes \iEnd A_N
    \]
    which give the first claim.

    For the second claim note that $A_{N,M}^*\otimes_{A_M} A_{N,M} \cong A_N$ in $\lrep{\left(A_N\otimes A_N^{\op}\right)} G$, which can be seen applying restriction to $\tilde{G}$ as above and considering isomorphisms
    \[
    A_{N,M}^*\otimes_{A_M} A_{N,M}\cong \iHom(M,N)\otimes_{\iEnd M} \iHom(N,M) \cong \iEnd N= A_N.
    \]
    Similarly, $A_{N,M}\otimes_{A_N} A_{N,M}^* \cong A_M$ in $\lrep{\left(A_M\otimes A_M^{\op}\right)} G$. Hence the functor
    \[
    A_{N,M}^*\otimes_{A_M} - \colon  \lrep{A_M} G \to \lrep{A_N} G
    \]
    gives the inverse equivalence.
\end{proof}

\begin{remark}
    The above lemma says that the class of $A_M$ in the $G$-equivariant Brauer group of $k$ depends only on the character of the $Z$-action on $M$ and not on the representation $M$ itself.
\end{remark}

\begin{lemma}\label{lem:character_rep_exist}
    For every $\theta\in \rfaktor{X^*(Z)}{\Gamma}$ the category $(\rep \tilde{G})_{\theta}$ is non-zero.
\end{lemma}
\begin{proof}
    Consider the regular representations $k[\tilde{G}]$ of $\tilde{G}$ and $k[Z]$ of $Z$. It follows from \cite[Corollary~4.13]{Mi17} that $k[Z]$ contains all finite dimensional irreducible $k$-representations of $Z$, in particular, there is $ \{0\}\neq N\subseteq  k[Z]$ such that $N\in (\rep Z)_{\theta}$. Note that the inclusion $Z\subseteq \tilde{G}$ yields a surjection $\res^{\tilde{G}}_Z k[\tilde{G}] \to k[Z]$ of $k$-representations of $Z$ which admits a splitting $s\colon k[Z] \to \res^{\tilde{G}}_Z k[\tilde{G}] $ by \cite[Theorem~12.30]{Mi17}. Let $M\subseteq k[\tilde{G}]$ be a finite-dimensional $\tilde{G}$-subrepresentation of $k[\tilde{G}]$ containing $s(N)$ \cite[Corollary~4.8]{Mi17}. We have $s(N)\subseteq M_{\theta}$ thus $0\neq M_\theta\in (\rep \tilde{G})_{\theta}$.
\end{proof}

The following is a generalization of \cite[Proposition~1.9]{Ela09} to groups of multiplicative type.
\begin{proposition} \label{prop:Elagin}
     Let $\lambda\in X^*(Z)$, put $\Gamma_{\lambda}:=\mathrm{Stab}_{\Gamma} \lambda \leqslant \Gamma$ and let $k\subseteq k(\lambda)\subseteq L$ be the corresponding intermediate field extension. Put
    \[
    \theta:=\{\lambda^\gamma\,|\, \gamma\in \Gamma\}\in \rfaktor{X^*(Z)}{\Gamma}.
    \]
    Then the following holds.
    \begin{enumerate}
        \item The functor of restriction of scalars
        \[
        (r_{k(\lambda)})_*\colon (\lrep{k(\lambda)} \tilde{G})_{\{\lambda\}} \to \rep \tilde{G}
        \]
        is fully faithful with the essential image $(\rep \tilde{G})_{\theta}$.
        \item Let $0\neq M\in (\lrep{k(\lambda)} \tilde{G})_{\{-\lambda\}}$, and consider $A_M\in \CSAlg_G k(\lambda)$ as in Definition~\ref{def:Azumaya_equiv}. Then
        \[
        \Phi\colon \lrep{A_M} G \to \rep \tilde{G},\quad W\mapsto  (r_{k(\lambda)})_*\left(\res^{G}_{\tilde{G}}W \otimes_{\res^{G}_{\tilde{G}} A_M^\op} M^*\right),
        \]
        is a fully faithful functor with the essential image $(\rep \tilde{G})_{\theta}$.
    \end{enumerate}

\end{proposition}
\begin{proof}
    Put $r_*:=(r_{k(\lambda)})_*$.
    
    1. Let $\mc{T}$ be a system of representatives of left cosets $\lfaktor{\Gamma_{\lambda}}{\Gamma}$. For $M\in \lrep{k(\lambda)} \tilde{G}$ we have 
    \[
    r_*M\otimes L \cong \bigoplus_{\gamma\in \mc{T}} (M \otimes_{k(\lambda)} L)^{\gamma}
    \]
    with $(M \otimes_{k(\lambda)} L)^{\gamma} \in ((\lrep{L} \tilde{G})_{\lambda})^{\gamma}= (\lrep{L} \tilde{G})_{{\lambda}^{\gamma}}$, hence $r_* M\in (\rep \tilde{G})_{\theta}$.
    
    Let $M,N\in (\lrep{k(\lambda)} \tilde{G})_{\{\lambda\}} $, we need to show that the canonical homomorphism of $k$-vector spaces
    \[
    r_*\colon \Hom_{\lrep{k(\lambda)} \tilde{G}}(M,N)\to \Hom_{\rep \tilde{G}} (r_*M,r_*N)
    \]
    is an isomorphism. It suffices to do so after extending the scalars to $L$. We have
    \begin{multline*}
    \left(\Hom_{\lrep{k(\lambda)} \tilde{G}} (M,N)\right)\otimes L
    \cong \bigoplus_{\gamma\in \mc{T}} \left(\Hom_{\lrep{k(\lambda)} \tilde{G}} (M,N)\otimes_{k(\lambda)} L\right)^{\gamma} \cong\\
    \cong \bigoplus_{\gamma\in \mc{T}} \Hom_{\lrep{L} \tilde{G}} ((M\otimes_{k(\lambda)} L)^{\gamma},(N\otimes_{k(\lambda)} L)^{\gamma}) \cong \\
    \cong \Hom_{\lrep{L} \tilde{G}} (\bigoplus_{\gamma\in \mc{T}} (M\otimes_{k(\lambda)} L)^{\gamma},\bigoplus_{\gamma'\in \mc{T}}(N\otimes_{k(\lambda)} L)^{\gamma'}) \cong \\
    \cong \Hom_{\lrep{L} \tilde{G}} (r_*M\otimes L,r_*N \otimes L) \cong
    \Hom_{\rep \tilde{G}} (r_*M,r_*N)\otimes L,
    \end{multline*}
    where for the third isomorphism we used that 
    \[
    (M\otimes_{k(\lambda)} L)^{\gamma}\in (\lrep{L} \tilde{G})_{\lambda^{\gamma}},\quad (N\otimes_{k(\lambda)} L)^{\gamma'}\in (\lrep{L} \tilde{G})_{\lambda^{\gamma'}},
    \]
    and orthogonality of $(\lrep{L} \tilde{G})_{\lambda^{\gamma}}$ and $(\lrep{L} \tilde{G})_{\lambda^{\gamma'}}$ for $\lambda^\gamma\neq \lambda^{\gamma'}$. Thus the functor is fully faithful.

    It remains to show that the functor is essentially surjective onto $(\rep \tilde{G})_{\theta}$. For $M\in (\rep \tilde{G})_{\theta}$ we have
    \[
    \left(\End_{\rep \tilde{G}} M\right)\otimes L \cong \End_{\lrep{L} \tilde{G}} \left(M\otimes L\right) \cong \bigoplus_{\lambda'\in \theta} \End_{\lrep{L} \tilde{G}} \left(M\otimes L\right)_{\lambda'}.
    \]
    The embeddings $L\subseteq \End_{\lrep{L} \tilde{G}} \left(M\otimes  L\right)_{\lambda'}$ as scalar multiplication induces a $\Gamma$-equivariant embedding 
    \[
    k(\lambda)\otimes L\cong \bigoplus_{\lambda'\in \theta} L \subseteq \bigoplus_{\lambda'\in \theta} \End_{\lrep{L} \tilde{G}} \left(M\otimes L\right)_{\lambda'} \cong \left(\End_{\rep \tilde{G}} M\right)\otimes L,
    \]
    which descends to the embedding $k(\lambda)\subseteq \End_{\rep \tilde{G}} M$. Hence $M$ belongs to the essential image of $r_*$.

    2. Using the first part of the proposition we may assume that $k(\lambda)=k$. Let $W_1,W_2\in \lrep{A_M} G$, then
    \begin{multline*}
    \Hom_{\rep \tilde{G}}(\left(\res^{G}_{\tilde{G}}W_1\right)\otimes_{\res^{G}_{\tilde{G}}A_M^\op}M^*, \left(\res^{G}_{\tilde{G}}W_2\right)\otimes_{\res^{G}_{\tilde{G}}A_M^\op}M^*)\cong\\
    \cong \Hom_{\lrep{\left(\res^{G}_{\tilde{G}}(A_M\otimes A_M^{\op})\right)} \tilde{G}}(M\otimes M^*, \res^{G}_{\tilde{G}} (W_1^*\otimes W_2))\cong\\
    \cong \Hom_{\lrep{\left(A_M\otimes A_M^{\op}\right)} G}(A_M, W_1^*\otimes W_2)\cong
    \Hom_{\lrep{A_M} G}(W_1, W_2).
    \end{multline*}
    Thus $\Phi$ is fully faithful. Since for $N\in \rep G$ we have $\res^{G}_{\tilde{G}} N\in (\rep \tilde{G})_{\{0\}}$ and $M^*\in (\rep \tilde{G})_{\{\lambda\}}$ then for $W\in \lrep{A_M} G$ we have $\Phi(W)\in (\rep \tilde{G})_{\{\lambda\}}$. Let $N\in (\rep \tilde{G})_{\{\lambda\}}$, then $Z$ acts trivially on $N\otimes M$ and $\tilde{G}$-action descends to a $G$-action giving rise to $A_{M^*,N}\in \lrep{(A_M\otimes A_N)} G$ such that $\res^G_{\tilde{G}}A_{M^*,N}\cong N\otimes M$. The isomorphisms
    \[
    N\cong N \otimes  (M\otimes_{\res^{G}_{\tilde{G}}A_M^{\op}} M^*) \cong \left(\res^{G}_{\tilde{G}} (A_{M^*,N})\right)\otimes_{\res^{G}_{\tilde{G}}A_M^\op}  M^*
    \]
    yield the claim.
\end{proof}

\begin{corollary} \label{cor:decomposition_equiv_Azumaya}
    Let $\mc{S}\subseteq X^*(Z)$ be a transversal for the $\Gamma$-action. For $\lambda\in \mc{S}$ consider the intermediate extension $k\subseteq k(\lambda)\subseteq L$ corresponding to $\mathrm{Stab}_{\Gamma} \lambda\leqslant \Gamma$, pick some $0\neq M({\lambda})\in (\lrep{k(\lambda)} \tilde{G})_{\{-\lambda\}}$ and denote $A_{\lambda}:=A_{M(\lambda)}\in \CSAlg_G k(\lambda)$. Then for $\lambda\in\mc{S}$ the functor
     \[
     \Phi_\lambda\colon \D^b(\lrep{A_\lambda} G) \to \D^b(\rep \tilde{G}),\quad W\mapsto (r_{k(\lambda)})_* \left(\res^{G}_{\tilde{G}}(W) \otimes_{\res^{G}_{\tilde{G}}{A^{\op}_\lambda}} M(\lambda)^*\right),
     \]
     is fully faithful with the essential image $\D^b(\rep \tilde{G})_{\theta_{\lambda}}$, $\theta_{\lambda}=\{\lambda^\gamma\,|\,\gamma\in\Gamma\}$, and there is an orthogonal decomposition
     \[
     \D^b(\rep \tilde{G})=\bigoplus_{\lambda\in \mc{S}} \Phi_\lambda(\D^b(\lrep{A_\lambda} G)).
     \]
\end{corollary}
\begin{proof}
    Follows from Proposition~\ref{prop:Elagin} and decompositions \refbr{eq:rep_blocks}, \refbr{eq:Db_rep_blocks}.
\end{proof}
 
\subsection{Relative decompositions for isogenous pairs}    
\label{sec:relative_split}
    In this section we show that a full relative \'etale-exceptional collection for a pair $\tilde{P}\leqslant \tilde{G}$ gives rise to a full relative separable-exceptional collection for an isogenous quotient pair $P\leqslant G$. Combining this with the results of Section~\ref{sec:qs_sc} we obtain such a collection for an arbitrary quasi-split group $G$.

   \begin{theorem} \label{thm:collection_isogenous}
        Let $\tilde{G}$ be a connected affine algebraic group over $k$, $\tilde{P}\leqslant \tilde{G}$ be a parabolic subgroup, $Z\leqslant \tilde{P}$ be a subgroup scheme of multiplicative type and suppose that $Z$ is central in $\tilde{G}$. Denote $G:=\tilde{G}/Z$, $P:=\tilde{P}/Z$. Suppose that $\D^b(\rep \tilde{P})$ admits a full $\tilde{G}$-relative \'etale-exceptional (resp. exceptional) collection. Then $\D^b(\rep P)$ admits a full $G$-relative separable-exceptional (resp. Azumaya-exceptional) collection.
    \end{theorem}
    \begin{proof}
        Let $X\in \lrep{A} \tilde{P}$ be a $\tilde{G}$-relative $A$-exceptional object with $A\in \Alg_{\tilde{G}} k$ being a commutative separable $\tilde{G}$-algebra. Forgetting the $\tilde{G}$-equivariant structure we have $A\cong \prod_{i=1}^m K_i$ for some separable finite field extensions $K_i/k$. Since $\tilde{G}$ is connected and the group of $k$-algebra automorphisms of $\prod_{i=1}^m K_i$ is finite, the action of $\tilde{G}$ on $A$ is necessarily trivial.
        
        Let $X_i\in \D^b(\lrep{A_i} \tilde{P})$, $1\le i\le n$, be a full $\tilde{G}$-relative \'etale-exceptional collection with $A_i\in\Alg_{\tilde{G}} k$ commutative separable $G$-algebras. By the above we may assume $A_i=K_i$ to be a finite separable field extension $K_i/k$ with the trivial action of $\tilde{G}$. Denote $(r_i)_*:=(r_{K_i})_*$ the corresponding functor of restriction of scalars. Since $X_i$ is $\tilde{G}$-relative $K_i$-exceptional, by adjunction we have $K_i\cong \End ( (r_i)_*X_i)$. Extending the scalars to $L$ we obtain
        \[
        \bigoplus_{\gamma \in \lfaktor{\Gamma_i}{\Gamma}} L^\gamma \cong K_i\otimes L\cong \End (((r_i)_*X_i)\otimes L) \cong \End (\bigoplus_{\gamma \in \lfaktor{\Gamma_i}{\Gamma}} (X_i\otimes_{K_i} L)^\gamma),
        \]
        where $\Gamma_i:=\Gal(L/K_i)$. Hence $\End ((X_i\otimes_{K_i} L)^\gamma)\cong L$ for every $\gamma\in\Gamma$, in particular, $X_i\otimes_{K_i} L$ is indecomposable.
        Then it follows from the decomposition~\refbr{eq:Db_rep_blocks} that $X_i\in \D^b(\lrep{K_i} \tilde{P})_{\{\lambda_i\}}$ for some $\lambda_i\in X^*(Z)^{\Gamma_i}$. Choose some $0\neq M_i\in (\lrep{K_i} \tilde{G})_{\{\lambda_i\}}$ which exists by Lemma~\ref{lem:character_rep_exist}, and denote $A_i:= A_{M_i}$. Then $X_i\otimes_{K_i} M_i^*\in  \D^b(\lrep{A^\op_i} \tilde{P})_{\{0\}}$ thus $X_i\otimes_{K_i} M_i^*\cong \res^P_{\tilde{P}} Y_i$ for some $Y_i\in \D^b(\lrep{A^\op_i} P)$.
        
        We claim that $\{Y_i\}_{1\le i\le n}$ is a full $G$-relative separable-exceptional collection. First we check that it is a $G$-relative separable-exceptional collection. 
        Using $\res^G_{\tilde{G}}\ind^G_P \cong \ind^{\tilde {G}}_{\tilde{P}}\res^P_{\tilde{P}}$ given by \cite[Proposition~I.6.11]{Jan03} and projection formula we obtain
        \begin{multline*}
        \res^G_{\tilde{G}}\ind^G_P\iHom_{\D^b(\rep P)}(Y_i,Y_j)\cong
        \ind^{\tilde {G}}_{\tilde{P}}\res^P_{\tilde{P}}\iHom_{\D^b(\rep P)}(Y_i,Y_j) \cong \\
        \cong
        \ind^{\tilde {G}}_{\tilde{P}}\iHom_{\D^b(\rep \tilde{P})}(X_i\otimes_{K_i} M_i^*,X_j\otimes_{K_j} M_j^*)
        \cong\\
        \cong
        \ind^{\tilde {G}}_{\tilde{P}}\left(\iHom_{\D^b(\rep \tilde{P})}(X_i,X_j)\otimes_{K_i\otimes K_j} \iHom_{\D^b(\rep \tilde{G})}(M_i^*,M_j^*)\right)
        \cong \\
        \cong
        \left(\ind^{\tilde {G}}_{\tilde{P}}\iHom_{\D^b(\rep \tilde{P})}(X_i,X_j)\right) \otimes_{K_i\otimes K_j} \iHom_{\D^b(\rep \tilde{G})}(M_i^*,M_j^*).
        \end{multline*}
        Since $\res^G_{\tilde{G}}$ is the fully faithful inclusion of $\D^b(\rep G)\cong \D^b(\rep \tilde{G})_{\{0\}}$, the above, together with the vanishing $\ind^{\tilde {G}}_{\tilde{P}}\iHom_{\D^b(\rep \tilde{P})}(X_i,X_j)=0$ for $1\le j< i\le n$, yields $\ind^{G}_{P}\iHom_{\D^b(\rep P)}(Y_i,Y_j)=0$ for $1\le j <i\le n$. Furthermore, for $i=j$ we have
        \begin{multline*}
        \res^G_{\tilde{G}}\ind^G_P\iEnd_{\D^b(\rep P)}(Y_i)\cong
        \left(\ind^{\tilde {G}}_{\tilde{P}}\iEnd_{\D^b(\rep \tilde{P})}(X_i)\right) \otimes_{K_i\otimes K_i} \iEnd_{\D^b(\rep \tilde{G})}(M_i^*) \cong\\
        \cong K_i\otimes_{K_i\otimes K_i}\iEnd_{\D^b(\rep \tilde{G})}(M_i^*) \cong \iEnd_{\D^b(\rep \tilde{G})}(M_i^*) \cong \res^G_{\tilde{G}} A_{i}^\op.
        \end{multline*}
        Thus the canonical homomorphism $A_i^\op\to \ind^G_P\iEnd_{\D^b(\rep P)} (Y_i)$ is an isomorphism. 

        In order to show the fullness of the collection recall that by the assumption we have
        \[
        \D^b(\rep \tilde{P})=\hull \left(\left\{\left.(r_i)_*\left(N\otimes_{K_i} X_i\right)\,\right|\, N\in \D^b(\lrep{K_i} \tilde{G})\right\}_{ 1\le i\le n} \right).
        \]
        Applying the decomposition~\refbr{eq:Db_rep_blocks} and taking into account that $X_i\in \D^b(\lrep{K_i} \tilde{P})_{\{\lambda_i\}}$, we obtain
        \[
        \D^b(\rep \tilde{P})_{\{0\}}=\hull \left(\left\{\left.(r_i)_*\left( N\otimes_{K_i} X_i\right)\,\right|\, N\in \D^b(\lrep{K_i} \tilde{G})_{\{-\lambda_i\}}\right\}_{ 1\le i\le n} \right).
        \]
        By Corollary~\ref{cor:decomposition_equiv_Azumaya} applied to $K_i$ in place of $k$ we have 
        \[
        \D^b(\lrep{K_i} \tilde{G})_{\{-\lambda_i\}} = \left\{  W\otimes_{A^{\op}_i} M_i^* \,|\, W\in \D^b(\lrep{A_i} G)\right\},
        \]
        hence
        \begin{multline*}
            \D^b(\rep \tilde{P})_{\{0\}}
            =\hull \left(\left\{\left.(r_i)_*\left((\res^{G}_{\tilde{G}}W\otimes_{A^{\op}_i} M_i^*) \otimes_{K_i} X_i\right)\,\right|\, W\in \D^b(\lrep{A_i} G)\right\}_{ 1\le i\le n} \right) = \\
            =\hull \left(\left\{\left.(r_i)_*\left(\res^{G}_{\tilde{G}}W\otimes_{A^{\op}_i} \res^{P}_{\tilde{P}}Y_i\right)\,\right|\, W\in \D^b(\lrep{A_i} G)\right\}_{ 1\le i\le n} \right)=
            \\
            =\res^{P}_{\tilde{P}}\left(\hull \left(\left\{\left.(r_i)_*\left(W\otimes_{A^{\op}_i} Y_i\right)\,\right|\, W\in \D^b(\lrep{A_i} G)\right\}_{ 1\le i\le n}\right) \right).
        \end{multline*}
        Recall that $\res^P_{\tilde{P}}\colon \D^b(\rep P)\to \D^b(\rep \tilde{P})_{\{0\}}$ is an equivalence, so the claim follows.

        After forgetting the $G$-action the algebra $A_i$ is isomorphic to $M_{n_i}(K_i)$, $n_i:=\dim_{K_i} M_i$, so $A_i\in\CSAlg_G K_i$. Thus a full $\tilde{G}$-relative exceptional collection gives rise to a full $G$-relative Azumaya-exceptional collection.
    \end{proof}

Below in this section $G$ is a quasi-split semisimple group over $k$, $T\leqslant B\leqslant P\leqslant G$ is a maximal torus, a Borel subgroup and a parabolic subgroup respectively, $\rho\colon \tilde{G}\to G$ is the simply connected cover, $\tilde{T}:=\rho^{-1} T$, $\tilde{B}:=\rho^{-1} B$, $\tilde{P}:=\rho^{-1} P$, $Z:=\ker \rho$. Recall that $Z$ is a finite group scheme of multiplicative type, $Z\leqslant \tilde{T}$ and $Z$ is central in $\tilde{G}$. We also freely use the notation and definitions of Section~\ref{sec:qs_sc} applied to $\tilde{T} \leqslant \tilde{B} \leqslant \tilde{P} \leqslant \tilde{G}$, in particular, $L$ is the splitting field of $\tilde{G}$, $\Gamma:=\Gal(L/k)$, $W$ is the Weyl group of $\tilde{G}$, $W^P$ is the set of minimal representatives for cosets $\rfaktor{W}{W_P}$, etc. 

\begin{definition}
    The embedding $Z\leqslant \tilde{T}$ induces a surjection $X^*(\tilde{T})\to X^*(Z)$. For $\lambda\in X^*(\tilde{T})$ we denote $\bar{\lambda}\in X^*(Z)$ its image under this surjection.
\end{definition}

\begin{lemma} \label{lem:Xp_character}
    Let $\le$ be a total order on $W$ refining the Bruhat order $\lebru$ and compatible with $\Gamma$-action in the sense of Definition~\ref{def:compaible_action}. Then for $v\in W^P$ we have $\hat{X}_{v,k(v)}\in \D^b(\lrep{k(v)} \tilde{P})_{\{\bar{e}_{v}\}}$ in the notation of Definitions~\ref{def:Xv_qs} and~\ref{def:blocks}.
\end{lemma}
\begin{proof}
    The $L$-representation $\mc{P}_v\cong \res^{\tilde{P}}_{\tilde{B}} (\hat{\mc{P}}_{v,k(v)}\otimes_{k(v)} L)$ of $B$ has 1-dimensional socle given by the weight ${-e_{v}}\in X^*(\tilde{T})$ \cite[\S~1.3]{vdK89}, hence it is indecomposable and $\mc{P}^*_v \in (\lrep{L} \tilde{B})_{\bar{e}_{v}}$. Then $\hat{\mc{P}}^*_{v}\in (\lrep{L} \tilde{P})_{\bar{e}_{v}}$ and hence $\hat{\mc{P}}^*_{v,k(v)}\in (\lrep{k(v)} \tilde{P})_{\{\bar{e}_{v}\}}$. The claim follows by Lemma~\ref{lem:SOD_direct_sum}. 
\end{proof}

\begin{definition}
\label{def:equiv_Tits}
    Let $\lambda\in X^*(\tilde{T})_+$ be a dominant weight, $ K:=k(\lambda):=L^{\operatorname{Stab}_\Gamma \lambda}$ be its field of definition and $K_\lambda\in \lrep{K} \tilde{B}$ be the $1$-dimensional $K$-representation of weight $\lambda$. Denote
    \[
    \nabla_\lambda:=\ind_{\tilde{B}}^{\tilde{G}} K_{\lambda} \in \lrep{K} \tilde{G}.
    \]
    It follows from \cite[Proposition~II.2.2.b)]{Jan03} that $(\nabla_\lambda\otimes_K L)\in (\lrep{L} \tilde{G})_{\bar{\lambda}}$, hence $\nabla_\lambda\in (\lrep{K} \tilde{G})_{\{\bar{\lambda}\}}$. We denote
    \[
    A_\lambda:=A_{\nabla_\lambda}\in \CSAlg_G K
    \]
    in the notation of Definition~\ref{def:Azumaya_equiv}, this is the $G$-equivariant central simple algebra over $K$ satisfying $\res^G_{\tilde{G}} A_\lambda \cong \iEnd \nabla_\lambda$. We refer to $A_\lambda$ as an \textit{equivariant Tits algebra} since the separable algebras introduced in \cite[\S~4]{Tits71} may be recovered from these ones as $\mc{A}_{\lambda}^{\mathcal{E}}$ (see Definition~\ref{def:descent} for the notation) twisting by an appropriate $G$-torsor $\mc{E}$, see \cite[Lemma~3.3]{Pan94} for the comparison.  
    
    Let $\lambda_1,\lambda_2\in X^*(\tilde{T})_+$ be such that $\bar{\lambda}_1 = \bar{\lambda}_2 \in X^*(Z)$ and $k(\lambda_1)=k(\lambda_2)$ (note that the second equality in general does not follow from the first one), then Lemma~\ref{lem:Azumaya_equiv} yields that the algebras $A_{\lambda_1}$ and $A_{\lambda_2}$ are equivariantly Brauer equivalent and there is an equivalence of categories $\lrep{A_{\lambda_1}} G \cong \lrep{A_{\lambda_2}} G$.
    
    Below we will be interested in the equivariant Tits algebras associated to the Steinberg weights (see Definition~\ref{def:Steinberg}) which we denote
    \[
    A_v:=A_{ve_v} \otimes_{k(ve_v)} k(v) \cong A_{\nabla_{ve_v}\otimes_{k(ve_v)} k(v)} \in\CSAlg_G k(v).
    \]
    Note that by Lemma~\ref{lem:weight_defined} we have $k(ve_v)\subseteq k(e_v)=k(v)$ in the notation of Definition~\ref{def:kv}, and in general the inclusion may be strict.
\end{definition}

\begin{definition} \label{def:Yqs}
    Let $\le$ be a total order on $W$ extending the Bruhat order $\lebru$ and compatible with $\Gamma$-action, and for $v\in W^P$ let $\hat{X}_{v,k(v)}\in \D^b(\lrep{k(v)} \tilde{P})$ be as in Definition~\ref{def:Xv_qs}. Then by Lemma~\ref{lem:Xp_character} and considerations from Definition~\ref{def:equiv_Tits} we have 
    \[
    \hat{X}_{v,k(v)}\otimes_{k(v)} \nabla_{e_v}^*\in \D^b(\lrep{k(v)} \tilde{P})_{\{0\}}\cong \D^b(\lrep{k(v)} P),
    \]
    giving rise to $\hat{Y}_v\in \D^b(\lrep{A^{\op}_v} P)$ characterized by the property
    \[
    \res^{P}_{\tilde{P}} \hat{Y}_v \cong \hat{X}_{v,k(v)}\otimes_{k(v)} \nabla_{e_v}^*.
    \]
\end{definition}

\begin{theorem} \label{thm:SOD_general}
    Let $\le$ be a total order on $W$ extending the Bruhat order $\lebru$ and compatible with $\Gamma$-action, and let $\{v_1,v_2,\hdots, v_m\}\subseteq W^P$ be a transversal  for the $\Gamma$-action enumerated in such a way that $v_m<v_{m-1}<\hdots<v_1$. Denote $A_i:=A_{v_i}$ and 
    \[
    \hat{Y}_i:=\hat{Y}_{v_i}\in \D^b(\lrep{A^{\op}_i} P)
    \]
    in the notation of Definition~\ref{def:Yqs}.  Then $\{\hat{Y}_i\}_{1\le i\le m}$ is a full $G$-relative separable-exceptional collection, the functors 
    \[
    \Phi_i\colon \D^b(\lrep{A_i} G) \to \D^b(\rep P), \quad M\mapsto (r_{k(v_i)})_*(M\otimes_{A^{\op}_i} \hat{Y}_i),
    \]
    are fully faithful and there is a semiorthogonal decomposition
        \[
        \D^b(\rep P) = \langle \Phi_1(\D^b(\lrep{A_1} G)), \Phi_2(\D^b(\lrep{A_2} G)), \hdots, \Phi_{m}(\D^b(\lrep{A_m} G))\rangle.
        \]
\end{theorem}
\begin{proof}
    Theorem~\ref{thm:SOD_rep_qs} says that $\{\hat{X}_i\}_{1\le i\le m}$ that are used to define $\{\hat{Y}_i\}_{1\le i\le m}$ form a full $\tilde{G}$-relative \'etale-exceptional collection in $\D^b(\rep \tilde{P})$. The proof of Theorem~\ref{thm:collection_isogenous} yields that $\{\hat{Y}_i\}_{1\le i\le m}$ is a full $G$-relative separable-exceptional collection, and Proposition~\ref{prop:SOD_from_collection_rep} gives the full faithfulness of functors and the semiorthogonal decomposition.
\end{proof}

\begin{remark} \label{rem:equiv_Tits_prop}
    In the notation of Theorem~\ref{thm:SOD_general} the following holds.
    \begin{enumerate}
        \item $A_i\in \CSAlg_G k(v_i)$.
        \item If $G$ is split, then the collection $\{\hat{Y}_i\}_{1\le i\le m}$ is $G$-relative Azumaya-exceptional.
        \item $[W:W_P]=\sum_{i=1}^m [k(v_i):k]$.
    \end{enumerate}
    All these properties are immediate from the definitions.
\end{remark}

\begin{remark} \label{rem:additional_orth}
    The collection $\{\hat{Y}_i\}_{1\le i\le m}$ depends on the choice of the total order $\le$. Making the choice as in Proposition~\ref{prop:more_orthogonality} one obtains additional orthogonality properties as in loc.cit.
\end{remark}

\newcommand{\etalchar}[1]{$^{#1}$}

\Addresses


\begin{thebibliography}{BLM{\etalchar{+}}21}

\bibitem[Ana12]{Ana12}
A.~Ananyevskiy,
\emph{On the algebraic $K$-theory of some homogeneous varieties},
Doc. Math. 17 (2012), 167--193

\bibitem[AG60]{AG60}
M.~Auslander, O.~Goldman,
\emph{The Brauer Group of a Commutative Ring},
Trans. Am. Math. Soc. 97:3 (1960), 367--409

\bibitem[Bae16]{Ba16}
S.~Baek,
\emph{Semiorthogonal decompositions for twisted Grassmannians},
Proc. Am. Math. Soc. 144 (2016),1--5

\bibitem[BLM{\etalchar{+}}21]{BLMNPS21}
A.~Bayer, M.~Lahoz, E.~Macr\`i, H.~Nuer, A.~Perry, P.~Stellari,
\emph{Stability conditions in families},
Publ. Math. IHES 133 (2021), 157--325

\bibitem[Bei79]{B79}
A.~A.~Beilinson,
\emph{Coherent sheaves on $\mathbb{P}^n$ and problems of linear algebra},
Funct. Anal. Appl. 12 (1979), 214--216

\bibitem[BZFN10]{BZFN10}
D.~Ben-Zvi, J.~Francis, D.~Nadler, 
\emph{Integral Transforms and Drinfeld Centers in Derived Algebraic Geometry},
J. Amer. Math. Soc. 23 (2010), 909--966

\bibitem[Ber09]{Ber09}
M.~Bernardara
\emph{A semiorthogonal decomposition for Brauer–Severi schemes},
Math. Nachr. 282:10 (2009), 1406--1413

\bibitem[Blu12]{Bl12}
M.~Blunk,
\emph{A derived equivalence for some twisted projective homogeneous varieties},
arXiv:1204.0537

\bibitem[BK89]{BK89}
A.~Bondal, M.~Kapranov,
\emph{Representable functors, Serre functors, and mutations},
Izv. Akad. Nauk SSSR Ser. Mat. 53 (1989), 1183--1205; English transl., Math. USSR-Izv. 35 (1990), 519--541

\bibitem[BO01]{BO01}
A.~Bondal, D.~Orlov,
\emph{Reconstruction of a variety from the derived category and groups of autoequivalences},
Comp. Math. 125:3 (2001), 327--344

\bibitem[BO02]{BO02}
A.~Bondal, D.~Orlov,
\emph{Derived categories of coherent sheaves},
Proc. Internat. Congress of Mathematicians (Beijing, 2002), vol. II, Higher Ed. Press, Beijing (2002): 47--56

\bibitem[BVdB03]{BVdB03}
A.~Bondal, M.~Van~den~Bergh,
\emph{Generators and representability of functors in commutative and noncommutative geometry},
Moscow Math. J. 3 (2003), 1--36

\bibitem[BDG17]{BDG17}
I.~Burban, Y.~Drozd, V.~Gavran,
\emph{Minors and resolutions of non-commutative schemes},
Eur. J. Math. 3:2 (2017), 311--341

\bibitem[Cal00]{Cal00}
A.~C\v ald\v araru,
\emph{Derived Categories of Twisted Sheaves on Calabi-Yau Manifolds},
Ph.D. thesis, Cornell University (2000), available at
\url{https://people.math.wisc.edu/~caldararu/}

\bibitem[DRC25]{DRC25}
A.~Dhillon, S.~Roy-Chowdhury,
\emph{Semiorthogonal decompositions for generalised Severi--Brauer schemes},
J. Algebra 680:15 (2025), 70--95

\bibitem[EGA II]{EGA2}
A.~Grothendieck, 
\emph{\'El\'ements de g\'eom\'etrie alg\'ebrique II. \'Etude globale \'el\'ementaire de quelques classes de morphismes},
IHES Publ. Math. No. 11, 1961

\bibitem[EGA IV${}_3$]{EGA43}
A.~Grothendieck, 
\emph{\'El\'ements de g\'eom\'etrie alg\'ebrique IV: \'Etude locale des sch\'emas et des morphismes de sch\'emas, Troisi\`eme partie},
IHES Publ. Math. No. 28, 1967

\bibitem[Ela09]{Ela09}
A.~D.~Elagin,
\emph{Semiorthogonal decompositions of derived categories of equivariant coherent sheaves},
Izv. Math., 73:5 (2009), 893--920

\bibitem[Fon24]{Fon24}
A.~V.~Fonarev,
\emph{Derived categories of Grassmannians: a survey},
Russian Math. Surveys 79:5 (2024), 807--845

\bibitem[Gro68]{Gro68}
A.~Grothendieck,
\emph{Le groupe de Brauer. I. Alg\`ebres d'Azumaya et interpr\'etations diverses},
in Dix Expos\'es sur la Cohomologie des Sch\'emas, North-Holland, Amsterdam, 1968

\bibitem[HNR19]{HNR19}
J.~Hall, A.~Neeman, D.~Rydh,
\emph{One positive and two negative results for derived categories of algebraic stacks},
J. Inst. Math. Jussieu 18:5 (2019), 1087--1111

\bibitem[Hum92]{Hu92}
J.~E.~Humphreys,
\emph{Reflection Groups and Coxeter Groups},
Cambridge University Press, 1992

\bibitem[Huy06]{Huy06}
D.~Huybrechts,
\emph{Fourier--Mukai transforms in algebraic geometry},
Oxford Mathematical Monographs, 2006

\bibitem[Jan03]{Jan03}
J.~C.~Jantzen,
\emph{Representations of algebraic groups},
2nd edition, Mathematical Surveys and Monographs, 107, 2003

\bibitem[Kap84]{Kap84}
M.~M.~Kapranov,
\emph{Derived category of coherent sheaves on Grassmann manifolds},
Izv. Akad. Nauk SSSR Ser. Mat. 48:1 (1984), 192--202

\bibitem[Kap88]{Kap88}
M.~M.~Kapranov,
\emph{On the derived categories of coherent sheaves on some homogeneous spaces},
Invent. Math. 92:3 (1988), 479--508

\bibitem[KMRT98]{KMRT98}
M.-A.~Knus, A.~Merkurjev, M.~Rost, J.-P.~Tignol,
\emph{The book of involutions,}
Colloquium Publications, vol. 44, AMS, 1998

\bibitem[KO74]{KO74}
M.-A.~Knus, M.~Ojanguren,
\emph{Th\'eorie de la descente et alg\`ebres d'Azumaya},
Lecture Notes in Mathematics, 389. Springer-Verlag, Berlin, New York, 1974

\bibitem[Kuz06]{Kuz06}
A.~Kuznetsov,
\emph{Hyperplane sections and derived categories},
Izv. Ross. Akad. Nauk Ser. Mat. 70:3 (2006), 23--128

\bibitem[Kuz08]{Kuz08}
A.~Kuznetsov,
\emph{Derived categories of quadric fibrations and intersections of quadrics},
Adv. Math. 218:5 (2008), 1340--1369

\bibitem[Kuz11]{Kuz11}
A.~Kuznetsov,
\emph{Base change for semiorthogonal decompositions},
Compos. Math. 147:3 (2011), 852--876

\bibitem[Kuz14]{Kuz14}
A.~Kuznetsov,
\emph{Semiorthogonal decompositions in algebraic geometry},
In Proceedings of the International Congress of Mathematicians—Seoul 2014. Vol. II, 635--660. Kyung Moon Sa, Seoul, 2014

\bibitem[MR16]{MR16}
C.~Mautner, S.~Riche,
\emph{On the exotic $t$–structure in positive characteristic},
Int. Math. Res. Not. 2016:18 (2016), 5727--5774

\bibitem[Mil80]{Mi80}
J.~S.~Milne,
\emph{\'Etale Cohomology},
Princeton Mathematical Series 33, Princeton University Press, 1980

\bibitem[Mil17]{Mi17}
J.~S.~Milne,
\emph{Algebraic groups. The theory of group schemes of finite type over a field},
Cambridge Studies in Advanced Mathematics, 170. Cambridge University Press, Cambridge, 2017

\bibitem[Nov23]{Nov23}
S.~Novakovi\'c,
\emph{On non-existence of full exceptional collections on some relative flags},
Rocky Mountain J. Math. 53:6 (2023), 1953--1964

\bibitem[Orl93]{O93}
D.~O.~Orlov
\emph{Projective bundles, monoidal transformations, and derived categories of coherent sheaves}
Russian Acad. Sci. Izv. Math. 41:1 (1993), 133--141

\bibitem[Orl20]{O20}
D.~Orlov,
\emph{Finite-dimensional differential graded algebras and their geometric realizations}
Advances in Mathematics 366 (2020), 107096

\bibitem[Pan94]{Pan94}
I.~Panin,
\emph{On the algebraic $K$-theory of twisted flag varieties}, 
K-Theory 8 (1994), 541--585

\bibitem[Sam07]{Sam07}
A.~Samokhin,
\emph{Some remarks on the derived categories of coherent sheaves on homogeneous spaces},
J. London Math. Soc. 76:1 (2007), 122--134

\bibitem[SvdK24]{SvdK24}
A.~Samokhin, W.~van~der~Kallen, 
\emph{Highest weight category structures on $Rep(B)$ and  full exceptional collections on generalized flag varieties over $\mathbb Z$}, 
arXiv:2407.13653.

\bibitem[Ser97]{Serre97}
J.-P.~Serre, \emph{Galois Cohomology},
Springer-Verlag Berlin 1997

\bibitem[Ste75]{St75}
R.~Steinberg,
\emph{On a theorem of Pittie},
Topology 14 (1975), 173--177

\bibitem[Tho87]{Th87}
R.~W.~Thomason,
\emph{Equivariant resolution, linearization, and Hilbert’s fourteenth problem over arbitrary base schemes},
Adv. in Math. 65 (1987), 16--34

\bibitem[Tho97]{Th97}
R.~W.~Thomason,
\emph{The classification of triangulated subcategories},
Comp. Math. 105:1 (1997), 1--27

\bibitem[Tit71]{Tits71}
J.~Tits,
\emph{Repr\'esentations lin\'eaires irr\'eductibles d'un groupe r\'eductif sur un corps quelconque},
J. Reine Angew. Math. 247 (1971), 196--220

\bibitem[vdK89]{vdK89}
W.~van~der~Kallen,
\emph{Longest weight vectors and excellent filtrations},
Math. Z. 201 (1989), 19--31

\bibitem[vdK93]{vdK93}
W.~van~der~Kallen,
\emph{Lectures on Frobenius splittings and $B$-modules}.
Notes by S.~P.~Inamdar. Published for the Tata Institute of Fundamental Research, Bombay; by Springer-Verlag, Berlin, 1993

\bibitem[Vis07]{Vi07}
A.~Vistoli,
\emph{Notes on Grothendieck topologies, fibered categories and descent theory},
arXiv:0412512v4
 
\end{thebibliography}
\end{document}